\documentclass[11pt, leqno]{article}

\usepackage{amssymb,amsmath,amsthm,mathrsfs}
\usepackage{graphicx}
\usepackage{times}
\usepackage{natbib}

\leftmargini=\baselineskip

\newtheoremstyle{localthm}
	{5pt} 
	{5pt} 
	{\sl} 
	{} 
	{\bf} 
	{{\rm.}} 
	{.7em} 
	{} 

\theoremstyle{localthm}
\newtheorem{Theorem}{Theorem}
\newtheorem{Corollary}[Theorem]{Corollary}
\newtheorem{Lemma}[Theorem]{Lemma}

\newtheoremstyle{localrem}
	{5pt} 
	{5pt} 
	{\rm} 
	{} 
	{\bf} 
	{{\rm.}} 
	{.7em} 
	{} 

\theoremstyle{localrem}
\newtheorem{Remark}[Theorem]{Remark}
\newtheorem{Definition}[Theorem]{Definition}
\newtheorem{Example}[Theorem]{Example}

\newcommand{\Ex}{\operatorname{\mathbb{E}}}
\renewcommand{\Pr}{\operatorname{\mathbb{P}}}

\newcommand{\R}{\mathbb{R}}
\renewcommand{\d}{\mathrm{d}}
\newcommand{\Rbar}{\bar{\mathbb{R}}}

\newcommand{\BB}{\mathcal{B}}
\newcommand{\LL}{\mathcal{L}}
\newcommand{\XX}{\mathcal{X}}

\newcommand{\lest}{\le_{\rm st}}
\newcommand{\lelr}{\le_{\rm lr}}
\newcommand{\gelr}{\ge_{\rm lr}}

\newcommand{\Snw}{S^{\rm nw}}
\newcommand{\Sse}{S^{\rm se}}
\newcommand{\Kw}{K^{\rm w}}
\newcommand{\Ke}{K^{\rm e}}
\newcommand{\qw}{q^{\rm w}}
\newcommand{\qe}{q^{\rm e}}

\renewcommand{\d}{\mathrm{d}}
\newcommand{\supp}{\mathrm{supp}}

\begin{document}
\addtolength{\baselineskip}{0.15\baselineskip}

\title{On Stochastic Orders and Total Positivity}
\author{Lutz D\"umbgen and Alexandre M\"osching\\
	University of Bern and F.\ Hoffmann-La Roche Ltd, Basel}
\date{\today}

\maketitle

\begin{abstract}
The usual stochastic order and the likelihood ratio order between probability distributions on the real line are reviewed in full generality. In addition, for the distribution of a random pair $(X,Y)$, it is shown that the conditional distributions of $Y$, given $X = x$, are increasing in $x$ with respect to the likelihood ratio order if and only if the joint distribution of $(X,Y)$ is totally positive of order two (TP2) in a certain sense. It is also shown that these three types of constraints are stable under weak convergence, and that weak convergence of TP2 distributions implies convergence of the conditional distributions just mentioned.
\end{abstract}



\section{Introduction}

In nonparametric statistics, shape-constraints on distributions offer an interesting alternative to more traditional smoothness constraints. Specifically, consider a univariate regression setting with a generic observation $(X,Y) \in \R \times \R$. In many applications it is a natural constraint that the conditional distribution $Q(\cdot \,|\, x) := \LL(Y \,|\, X = x)$ is increasing in $x \in \R$ with respect to some order on the set of probability distributions on the real line. \cite{Moesching_Duembgen_2020} analyze estimators of $Q(\cdot \,|\, x)$ under the sole assumption that it is increasing in $x$ with respect to the usual stochastic order, and \cite{Moesching_Duembgen_2022} investigate the same estimation problem under the stronger constraint that $Q(\cdot \,|\, x)$ is increasing in $x$ with respect to the likelihood ratio order.

The likelihood ratio order has also received attention in nonparametric methods for two- or $k$-sample problems, for instance, \cite{Dardanoni_Forcina_1998}, \cite{Westling_etal_2022} and \cite{Hu_etal_2022}.

One goal of this paper is to review stochastic order, likelihood ratio order and their properties and relations in full generality, extending previous results presented by \cite{Shaked_Shantikumar_2007}. Here are precise definitions of these orders, where $Q_1, Q_2$ denote arbitrary probability distributions on the real line:
\begin{description}
\item[Stochastic order:] $Q_1 \lest Q_2$ if $Q_1((y,\infty)) \le Q_2((y,\infty))$ for all $y \in \R$.
\item[Likelihood ratio order:] $Q_1 \lelr Q_2$ if $Q_1, Q_2$ admit densities $g_1, g_2$ with respect to some dominating measure such that $g_2/g_1$ is increasing on $\{g_1 + g_2 > 0\}$.
\end{description}
It is known and will be explained later that $Q_1 \lelr Q_2$ implies that $Q_1 \lest Q_2$. Note that likelihood ratio ordering is familiar from discriminant analysis and mathematical statistics, see \cite{Karlin_Rubin_1956} and \cite{Lehmann_Romano_2005}. Suppose that $Q_1 \lelr Q_2$. If we observe $(Y,C) \in \R \times \{1,2\}$, where $\LL(Y \,|\, C = c) = Q_c$, then the posterior probability $\Pr(C = 2 \,|\, Y)$ is isotonic in $Y$. If we observe $Y \sim Q_\theta$ with an unknown parameter $\theta \in \{1,2\}$, then an optimal test of the null hypothesis that $\theta = 1$ versus the alternative hypothesis that $\theta = 2$ rejects the former for large values of $Y$. Furthermore, likelihood ratio ordering is a frequent assumption or implication of models in mathematical finance, see \cite{Beare_2015} and \cite{Jewitt_1991}.

The notion of likelihood ratio order and its relation to stochastic order is reviewed thoroughly in Section~\ref{Sec:LROrder}, showing that it defines a partial order on the set of all probability measures on the real line which is preserved under weak convergence. No restrictions on the distributions are needed. It is also explained that two distributions are likelihood ratio ordered if and only if the corresponding receiving operator characteristic (ROC) is a concave curve, and this is related to convexity of their ordinal dominance curve. The latter results generalize previous work, e.g.\ by \cite{Westling_etal_2022}.

Another goal of this paper is to fully understand the relationship between the distribution of a random pair $(X,Y)$ and order constraints on the conditional distributions $Q(\cdot \,|\, x)$, $x \in \R$. In Section~\ref{Sec:Bivariate} it is shown how to characterize $\LL(X,Y)$ such that $Q(\cdot \,|\, x)$ is increasing in $x$ with respect to stochastic order or likelihood ratio order. In particular, we explain the connection between likelihood ratio ordering and total positivity of order two (TP2) of bivariate distributions. Finally, it is shown that these properties of bivariate distributions are preserved under weak convergence, and that this implies convergence of the conditional distributions. To the best of our knowledge, most of the material in Section~\ref{Sec:Bivariate} is completely new.

All proofs are deferred to Section~\ref{Sec:Proofs}.

\section{Likelihood ratio order}
\label{Sec:LROrder}

Let $Q_1$ and $Q_2$ be probability measures on $\R$, equipped with its Borel $\sigma$-field $\BB$. Our first result provides various equivalent conditions for $Q_1 \lelr Q_2$, the first of which is the definition given in the introduction.

\begin{Theorem}
\label{Thm:EquivDef}
The following four conditions on $Q_1$ and $Q_2$ are equivalent:\\[1ex]
\textbf{(i)} There exist a $\sigma$-finite measure $\mu$ and densities $g_j = \d Q_j/\d\mu$ ($j = 1,2$) such that
\[
	g_2/g_1 \ \ \text{is isotonic on} \ \ \{g_1 + g_2 > 0\} .
\]
\textbf{(ii)} There exist a $\sigma$-finite measure $\mu$ and densities $g_j = \d Q_j/\d\mu$ ($j=1,2$) such that
\[
	g_1(y) g_2(x) \ \le \ g_1(x) g_2(y) \quad\text{whenever} \ x < y .
\]
\textbf{(iii)} For arbitrary sets $A,B \in \BB$ such that $A < B$ (element-wise),
\[
	Q_1(B)Q_2(A) \ \le \ Q_1(A)Q_2(B) .
\]
\textbf{(iv)} There exists a dense subset $D$ of $\R$ such that for all intervals $A = (x,y]$ and $B = (y,z]$ with endpoints $x < y < z$ in $D$,
\[
	Q_1(B)Q_2(A) \ \le \ Q_1(A)Q_2(B) .
\]
\end{Theorem}

\begin{Remark}
\cite{Shaked_Shantikumar_2007} restrict their attention to probability measures which are either discrete or dominated by Lebesgue measure. Their definition of likelihood ratio order corresponds to properties (i--ii) with $\mu$ being counting measure or Lebesgue measure on the real line.
\end{Remark}

\begin{Remark}[Weak convergence]
\label{Rem:Weak.convergence1}
An important aspect of Theorem~\ref{Thm:EquivDef} is that conditions~(iii-iv) do not involve an explicit dominating measure or explicit densities of $Q_1$ and $Q_2$. In particular, if $(Q_{n1})_n$ and $(Q_{n2})_n$ are sequences of distributions on $\R$ converging weakly to $Q_1$ and $Q_2$, respectively, and if $Q_{n1} \lelr Q_{n2}$ for all $n$, then $Q_1 \lelr Q_2$ as well. This follows from (iv) with $D$ being the set of all $x \in \R$ such that $Q_1(\{x\}) = 0 = Q_2(\{x\})$.
\end{Remark}

\begin{Remark}[Stochastic order]
\label{Rem:ST_LR}
The likelihood ratio order is stronger than the stochastic order in the sense that $Q_1 \lest Q_2$ if $Q_1 \lelr Q_2$. This follows immediately from condition~(iii) applied to $A:=(-\infty,y]$ and $B:=(y,\infty)$ for arbitrary $y \in \R$, because $Q_1(B)Q_2(A) = Q_1(B) - Q_1(B)Q_2(B)$ and $Q_1(A)Q_2(B) = Q_2(B) - Q_1(B)Q_2(B)$.
\end{Remark}

The reverse statement is false in general, but the likelihood ratio order of two distributions is tightly connected to stochastic order of domain-conditional distributions. 

\begin{Lemma}
\label{Lem:LR_iff_condST} The following four properties of $Q_1$ and $Q_2$ are equivalent:\\[1ex]
\textbf{(i)} \ $Q_1 \lelr Q_2$.\\
\textbf{(ii)} \ $Q_1(\cdot\,|\,C) \lelr Q_2(\cdot\,|\,C)$ for all sets $C \in \BB$ such that $Q_1(C),Q_2(C)>0$.\\
\textbf{(iii)} \ $Q_1(\cdot\,|\,C) \lest Q_2(\cdot\,|\,C)$ for all sets $C \in \BB$ such that $Q_1(C),Q_2(C)>0$.\\
\textbf{(iv)} \ There exists a dense subset $D$ of $\R$ such that $Q_1(\cdot \,|\, (x,y]) \lest Q_2(\cdot \,|\, (x,y])$ for arbitrary points $x < y$ in $D$ such that $Q_1((x,y]), Q_2((x,y]) > 0$.
\end{Lemma}

The next result shows that the relation $\lelr$ defines indeed a partial order on the space of arbitrary probability measures on the real line.

\begin{Lemma}
\label{Lem:LR_Partial_Order}
For arbitrary probability measures $Q_1, Q_2, Q_3$ on $\R$,
\begin{align*}
	&Q_1 \lelr Q_1
		&& \text{(reflexivity)} ; \\
	&Q_1 \lelr Q_2 \ \ \text{and} \ \ Q_2\lelr Q_1 \ \ \text{implies that} \ \ Q_1 = Q_2
		&& \text{(antisymmetry)} ; \\
	&Q_1 \lelr Q_2 \ \ \text{and} \ \ Q_2 \lelr Q_3 \ \ \text{implies that} \ \ Q_1 \lelr Q_3
		&& \text{(transitivity)} .
\end{align*}
\end{Lemma}

Another interesting aspect of the likelihood ratio order is its connection with the \textsl{receiver operating characteristic (ROC)} of a pair $(Q_1,Q_2)$, that means, the set
\[
	\mathrm{ROC}(Q_1,Q_2) \ := \ \bigl\{ \bigl( Q_1(H), Q_2(H) \bigr) : H \in \mathcal{H} \bigr\}
\]
with $\mathcal{H}$ denoting the set of left-bounded half-lines, augmented by $\emptyset$ and $\R$,
\[
	\mathcal{H}
	\ := \ \bigl\{ (y,\infty) : y \in \R \bigr\}
		\cup \bigl\{ [y,\infty) : y \in \R \bigr\}
		\cup \{\emptyset,\R\} .
\]
Note that the family $\mathcal{H}$ is totally ordered by inclusion. With the distribution function $G_j$ of $Q_j$ ($j=1,2$) one may also write
\begin{align*}
	&\mathrm{ROC}(Q_1,Q_2) \\
	& \ = \ \bigl\{ \bigl( 1 - G_1(y), 1 - G_2(y) \bigr) : y \in \Rbar \bigr\}
		\cup \bigl\{ \bigl( 1 - G_1(y\,-), 1 - G_2(y\,-) \bigr) : y \in \R \bigr\} ,
\end{align*}
where $\Rbar := [-\infty,\infty]$ and $G_j(-\infty) := 0$, $G_j(\infty) := 1$. Thus we take the freedom to refer to $\mathrm{ROC}(Q_1,Q_2)$ as the ROC \textsl{curve} of $Q_1$ and $Q_2$, imagining a (possibly non-continuous) curve within the unit square $[0,1]\times [0,1]$, connecting the points $(1,1)$ and $(0,0)$.

Obviously, the ROC curve is isotonic in the sense that if $(a_1,a_2), (b_1,b_2) \in \mathrm{ROC}(Q_1,Q_2)$, then $(a_1,a_2) \le (b_1,b_2)$ or $(a_1,a_2) \ge (b_1,b_2)$ component-wise. By means of Theorem~\ref{Thm:EquivDef}, one can easily show that likelihood ratio order is equivalent to a concavity property of the ROC curve.

\begin{Corollary}
\label{Cor:LR_ROC}
Two distributions $Q_1, Q_2$ satisfy $Q_1 \lelr Q_2$ if and only if $\mathrm{ROC}(Q_1,Q_2)$ is concave in the following sense: If $(a_1,a_2)$, $(b_1,b_2)$ and $(c_1,c_2)$ are three different points in $\mathrm{ROC}(Q_1,Q_2)$ with $a_1 \le b_1 \le c_1$ and $a_2 \le b_2 \le c_2$, then
\[
	\frac{b_2 - a_2}{b_1 - a_1} \ \ge \ \frac{c_2 - b_2}{c_1 - b_1} .
\]
\end{Corollary}

An object related to the ROC curve of $Q_1,Q_2$ is the \textsl{ordinal dominance curve} $H_{G_1,G_2} : [0,1] \to [0,1]$,
\[
	H_{G_1,G_2}(\alpha) \ := \ G_2^{}\bigl(G_1^{-1}(\alpha)\bigr) ,
\]
where $G_j^{-1}(\alpha) := \min\{y \in \Rbar : G_j^{}(y) \ge \alpha\}$. The next result has been shown by \cite{Lehmann_Rojo_1992} in the special case of $Q_2 \ll Q_1$ and $G_1, G_2$ being continuous and strictly increasing, where $Q_2 \ll Q_1$ means that $Q_2$ admits a density with respect to $Q_1$. \cite{Westling_etal_2022} proved the same result under the assumption that $Q_2 \ll Q_1$ with a density $\mathrm{d}Q_2 / \mathrm{d}Q_1$ which is continuous on the support of $Q_1$.

\begin{Corollary}
\label{Cor:Equiv_LR_ODC}
Suppose that $Q_2 \ll Q_1$. Then $Q_1 \lelr Q_2$ if and only if the ordinal dominance curve $H_{G_1,G_2}$ is convex on the image of $G_1$, i.e.\ the set $G_1(\Rbar) = \{G_1(y) : y \in \Rbar\}$.
\end{Corollary}

\begin{Remark}
The assumption that $Q_2 \ll Q_1$ is essential in Corollary~\ref{Cor:Equiv_LR_ODC}. For instance, if $Q_1=(\delta_0+\delta_1)/2$ with $\delta_y$ denoting Dirac measure at $y$ and $Q_2=\mathrm{Unif}(0,1)$, then $H_{G_1,G_2}(0)=H_{G_1,G_2}(1/2)=0$ and $H_{G_1,G_2}(1)=1$. In consequence, the ordinal dominance curve restricted to $G_1(\Rbar)=\{0,1/2,1\}$ is convex, but $Q_1$ is not smaller than $Q_2$ with respect to likelihood ratio order since $\mathrm{ROC}(Q_1,Q_2)=\{(0,0),(1,1)\}\cup\{(1/2,u):0\le u\le 1\}$ is not concave.
\end{Remark}

\begin{Example}
Let $Q_1 = \mathcal{N}(1,1)$ and $Q_2 = \mathcal{N}(3,6)$. Here $\log dQ_2/dQ_1(y)$ is strictly convex in $y \in \R$ with limit $\infty$ as $|y| \to \infty$, whence neither $Q_1 \lelr Q_2$ nor $Q_2 \lelr Q_1$. The left panel of Figure~\ref{fig:ROC_ODC} shows the corresponding ROC curve $\mathrm{ROC}(Q_1,Q_2)$ which is obviously not concave. Now let us replace these Gaussian distributions with gamma distributions, $Q_1 = \mathrm{Gamma}(\alpha_1,\beta_1)$ and $Q_2 = \mathrm{Gamma}(\alpha_2,\beta_2)$, where $\mathrm{Gamma}(\alpha,\beta)$ denotes the gamma distribution with shape parameter $\alpha > 0$ and scale parameter $\beta > 0$. One can easily show that $\log dQ_2/dQ_1(y)$ is strictly increasing in $y > 0$ if $\alpha_1 \le \alpha_2$, $\beta_1 \le \beta_2$ and $(\alpha_1,\beta_1) \ne (\alpha_2,\beta_2)$. The right panel of Figure~\ref{fig:ROC_ODC} shows the concave ROC curve for $(\alpha_1,\beta_1) = (1,1)$ and $(\alpha_2,\beta_2) = (1.5,2)$, i.e.\ the first and second moments of $Q_1$ and $Q_2$ coincide with those of the Gaussian distributions considered before.
\end{Example}

\begin{figure}
\includegraphics[width=0.49\textwidth]{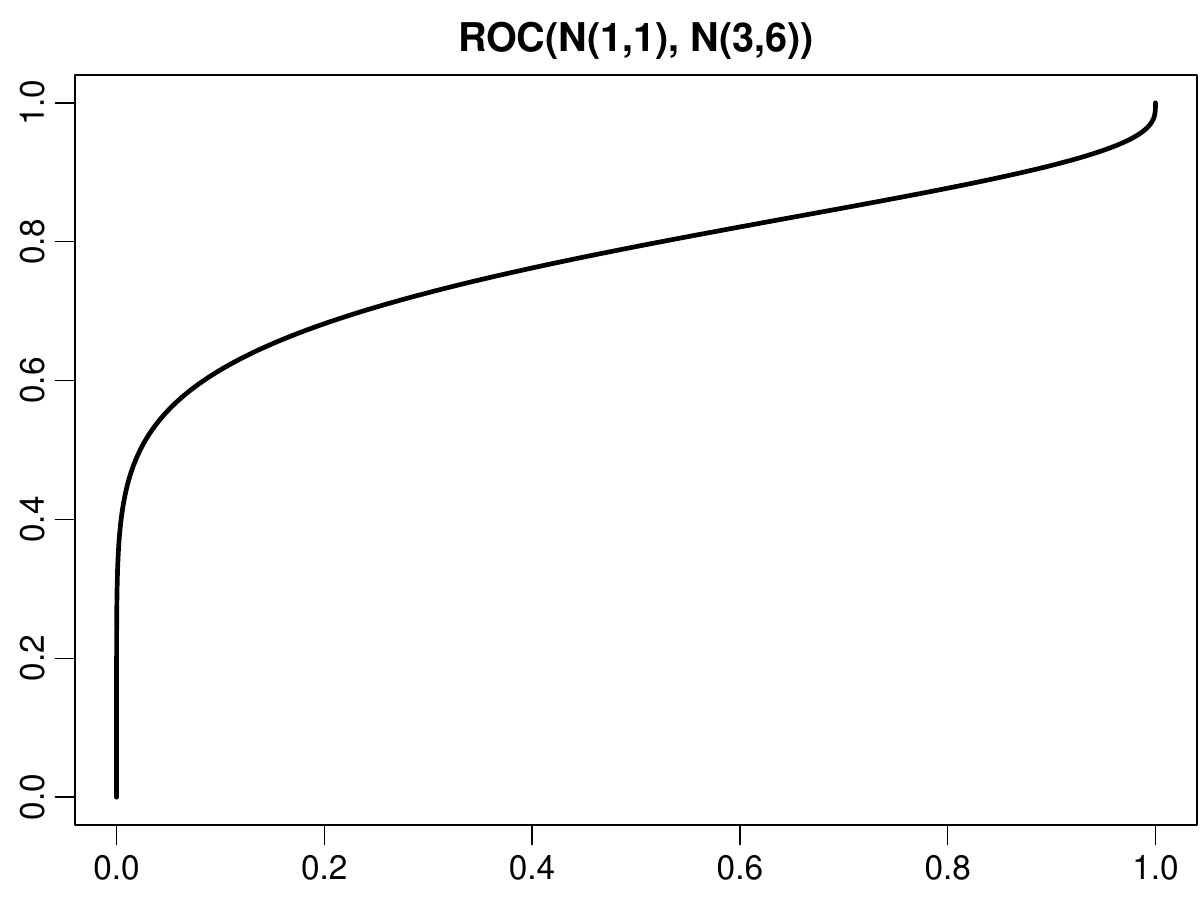}
\hfill
\includegraphics[width=0.49\textwidth]{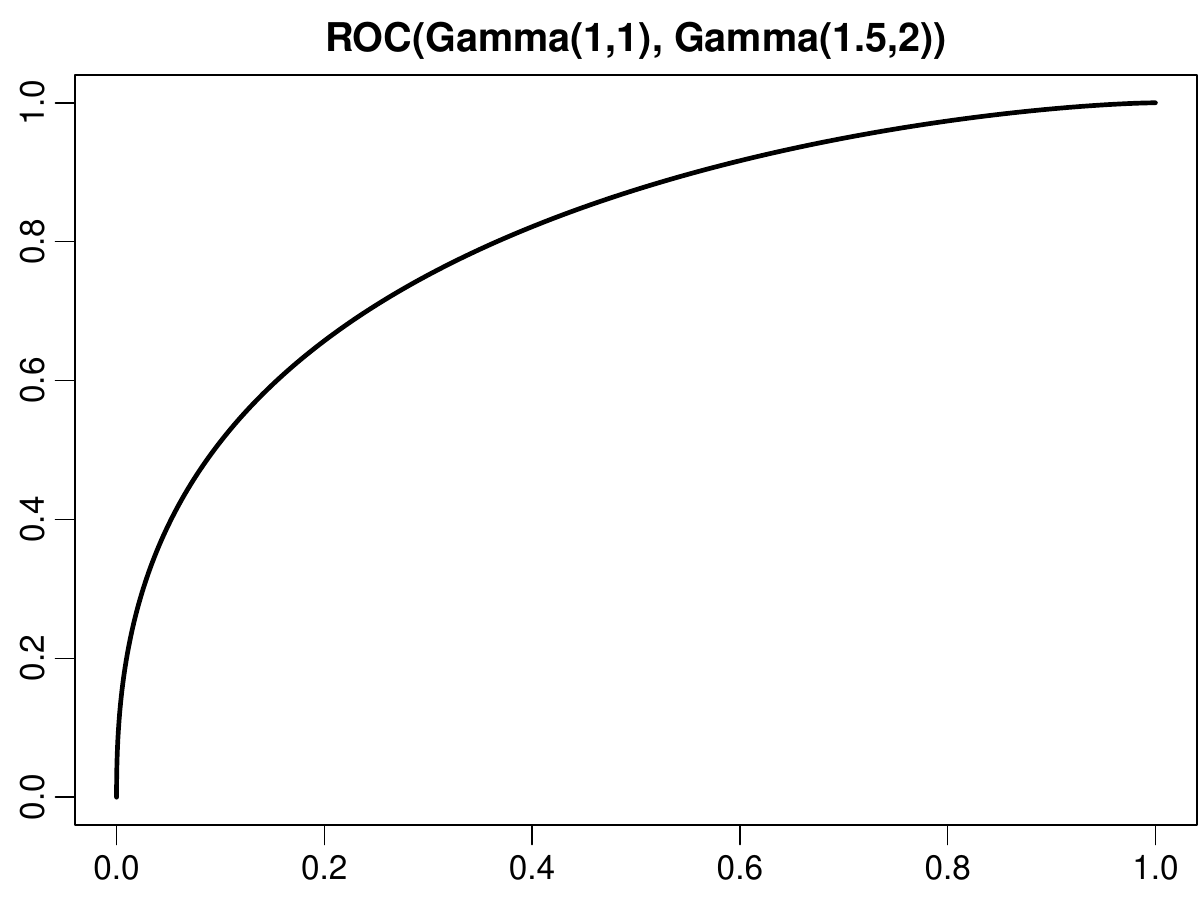}
\caption{ROC curves for two pairs $(Q_1,Q_2)$.}
\label{fig:ROC_ODC}
\end{figure}

\section{Bivariate distributions}
\label{Sec:Bivariate}

Throughout this section we consider a pair $(X,Y)$ of real-valued random variables with distribution $R = \LL(X,Y)$ and marginal distributions $P = \LL(X)$ and $Q = \LL(Y)$. We investigate order constraints on the conditional distributions of $Y$, given that $X$ lies in certain regions. The \textsl{range} of the random variable $X$ is defined as the set
\[
	\XX \ := \ \bigl\{ x \in \R :
		\Pr(X \le x), \Pr(X \ge x) > 0 \bigr\} .
\]
One can easily verify that $\XX$ is the smallest real interval such that $P(\XX) = \Pr(X \in \XX) = 1$.

\subsection{Conditional distributions and order constraints}
\label{Subsec:Conditional.order}

It is well-known from measure theory that the conditional distribution of $Y$, given $X$, may be described by a stochastic kernel $K : \R \times \BB \to [0,1]$. That is,
\begin{itemize}
\item for any fixed $x \in \R$, $K(x,\cdot)$ is a probability measure on $\BB$;
\item for any fixed $B \in \BB$, $K(\cdot,B)$ is measurable on $\R$;
\item for arbitrary $A,B \in \BB$,
\[
	R(A \times B) \ = \ \int_A K(x,B) \, P(\d x) .
\]
\end{itemize}
The following two theorems clarify under which conditions on the distribution $R$ of $(X,Y)$, the conditional distribution $K(x,\cdot)$ is isotonic in $x \in \XX$ with respect to stochastic order or likelihood ratio order. Interestingly, the proof is constructive, showing that
\[
	K(x, (y,\infty)) \ = \ \sup_{a < x\colon P((a,x]) > 0} \Pr(Y > y \,|\, a < X \le x)
\]
for $x \in \XX$ and $y \in \R$.

\begin{Theorem}[Stochastic order]
\label{Thm:Kernel.st}
The following three conditions are equivalent:
\begin{description}
\item[(i)] For arbitrary sets $A_1, A_2 \in \BB$ with $A_1 < A_2$ and real numbers $y$,
\begin{equation}
\label{ineq:kernel.st}
	R(A_1 \times (y,\infty)) P(A_2) \
	\le \ P(A_1) R(A_2 \times (y,\infty)) .
\end{equation}
\item[(ii)] There exists a dense subset $D$ of $\R$ such that for arbitrary numbers $x_0 < x_1 < x_2$ and $y$ in $D$, \eqref{ineq:kernel.st} holds true with $A_1 = (x_0,x_1]$ and $A_2 = (x_1,x_2]$.
\item[(iii)] The stochastic kernel $K$ may be constructed such that for arbitrary $x_1, x_2 \in \XX$ with $x_1 < x_2$,
\[
	K(x_1,\cdot) \ \lest \ K(x_2,\cdot) .
\]
\end{description}
\end{Theorem}

\begin{Remark}
Note that $P(A_j) = R(A_j \times (-\infty,y]) + R(A_j \times (y,\infty))$, so inequality~\eqref{ineq:kernel.st} is equivalent to
\begin{equation}
\label{ineq:kernel.st'}
	R(A_1 \times (y,\infty)) R(A_2 \times (-\infty,y]) \
	\le \ R(A_1 \times (-\infty,y]) R(A_2 \times (y,\infty)) .
\end{equation}
This underlines the connection to the next theorem about likelihood ratio order.
\end{Remark}

\begin{Theorem}[Likelihood ratio order]
\label{Thm:Kernel.lr}
The following three conditions are equivalent:
\begin{description}
\item[(i)] For arbitrary Borel sets $A_1, A_2, B_1, B_2$ with $A_1 < A_2$ and $B_1 < B_2$,
\begin{equation}
\label{ineq:Kernel.lr}
	R(A_1\times B_2) R(A_2\times B_1) \
	\le \ R(A_1\times B_1) R(A_2\times B_2) .
\end{equation}
\item[(ii)] There exists a dense subset $D$ of $\R$ such that for arbitrary numbers $x_0 < x_1 < x_2$ and $y_0 < y_1 < y_2$ in $D$, \eqref{ineq:Kernel.lr} holds true with $A_1 = (x_0,x_1]$, $A_2 = (x_1,x_2]$ and $B_1 = (y_0,y_1]$, $B_2 = (y_1,y_2]$.
\item[(iii)] The stochastic kernel $K$ may be constructed such that for arbitrary $x_1, x_2 \in \XX$ with $x_1 < x_2$,
\[
	K(x_1,\cdot) \ \lelr \ K(x_2,\cdot) .
\]
\end{description}
\end{Theorem}

\begin{Remark}
Condition~(i) in Theorems~\ref{Thm:Kernel.st} and \ref{Thm:Kernel.lr} is equivalent to the following statement about conditional distributions: Let $A_1, A_2$ be arbitrary Borel sets such that $A_1 < A_2$ and $P(A_1)$, $P(A_2)$ are strictly positive. Then,
\begin{align*}
	\LL(Y \,|\, X \in A_1) \
	&\lest \ \LL(Y \,|\, X \in A_2)
		\quad \text{(Theorem~\ref{Thm:Kernel.st})} , \\
	\LL(Y \,|\, X \in A_1) \
	&\lelr \ \LL(Y \,|\, X \in A_2)
		\quad \text{(Theorem~\ref{Thm:Kernel.lr})} .
\end{align*}
Condition~(ii) in Theorems~\ref{Thm:Kernel.st} and \ref{Thm:Kernel.lr} may be re-interpreted analogously.
\end{Remark}

\subsection{Total positivity of order two}
\label{Subsec:TP2}

Theorem~\ref{Thm:Kernel.lr} hints at a connection between likelihood ratio ordering and total positivity. A good starting point for the latter concept is the monograph of \cite{Karlin_1968}. Recall that a function $h : \R\times \R \to \R$ is called \textsl{totally positive of order two (TP2)} if for arbitrary real numbers $x_1 < x_2$ and $y_1 < y_2$,
\begin{equation}
\label{ineq:TP2.h}
	h(x_2,y_1) h(x_1,y_2) \ \le \ h(x_1,y_1) h(x_2,y_2) .
\end{equation}
Thus, conditions~(i) and (ii) of Theorem~\ref{Thm:Kernel.lr} may be interpreted as a distributional notion of total positivity of order two.

\begin{Definition}[Total positivity of order two of bivariate distributions]
A probability distribution $R$ on $\R \times \R$ is called \textsl{totally positive of order two (TP2)} if for arbitrary Borel sets $A_1 < A_2$ and $B_1 < B_2$,
\[
	R(A_2 \times B_1) R(A_1 \times B_2) \ \le \ R(A_1 \times B_1) R(A_2 \times B_2) .
\]
\end{Definition}

\begin{Remark}[TP2 of distributions and densities]
Suppose that the distribution of $(X,Y)$ has a density $h$ with respect to a product measure $\mu \otimes \nu$ on $\R \times \R$, where $\mu$ and $\nu$ are $\sigma$-finite measures on $\R$. If the function $h$ is TP2, then the distribution of $(X,Y)$ is TP2. To verify this, one has to express the probabilities in conditions~(i) or (ii) of Theorem~\ref{Thm:Kernel.lr} as integrals of $h$ with respect to $\mu \otimes \nu$.
\end{Remark}

When looking at the previous remark, one might think at first glance that a reverse statement is true with, say, $\mu = P$ and $\nu = Q$. But this is false. For instance, let $X$ be uniformly distributed on $[0,1]$, and let the conditional distribution of $Y$ given $X$ be given by the kernel $K$ with
\[
	K(x,\cdot) \ := \ \begin{cases}
		\mathrm{Unif}(0,1/3) & \text{if} \ x \le 1/3 , \\
		\delta_x & \text{if} \ 1/3 < x < 2/3 , \\
		\mathrm{Unif}(2/3,1) & \text{if} \ x \ge 2/3 .
	\end{cases}
\]
Then $Y$ follows the uniform distribution on $[0,1]$ too, and condition~(iii) of Theorem~\ref{Thm:Kernel.lr} is clearly satisfied, but the distribution of $(X,Y)$ has no density with respect to any product measure $\mu \otimes \nu$. Theorem~\ref{Thm:TP2} below provides a general statement, but let us start with a special case with a clean, positive answer which follows immediately from Theorem~\ref{Thm:Kernel.lr}: 

\begin{Corollary}
\label{Cor:DiscreteTP2}
Suppose that $(X,Y)$ has a discrete distribution with probability mass function $h$, i.e.\ $h(x,y) = \Pr(X = x, Y = y)$. Then the following two conditions are equivalent:
\begin{description}
\item[(i)] For arbitrary real numbers $x_1 < x_2$ with $\Pr(X = x_1), \Pr(X = x_2) > 0$,
\[
	K(x_1,\cdot) \ \lelr \ K(x_2,\cdot) .
\]
\item[(ii)] The function $h$ is TP2.
\end{description}
\end{Corollary}

To formulate our general result about total positivity and likelihood ratio order, we introduce a northwest and a southeast boundary for each $x \in \R$:
\begin{align*}
	\Snw(x) \
	:= \ &\min \bigl\{ y \in \Rbar : \Pr(X \le x, Y > y) = 0 \bigr\} , \\
	\Sse(x) \
	:= \ &\max \bigl\{ y \in \Rbar :  \Pr(X \ge x, Y < y) = 0 \bigr\} .
\end{align*}
One can easily verify that this defines two isotonic functions $\Snw, \Sse : \R \to \Rbar$. Moreover, on $\XX$, $\Snw > -\infty$ and $\Sse < \infty$.

\begin{Theorem}
\label{Thm:TP2}
Suppose that the distribution of $(X,Y)$ is TP2. Set
\[
	\XX_o \ := \ \bigl\{ x \in \XX : \Snw(x) \le \Sse(x) \bigr\} .
\]
Then the kernel $K$ can be constructed such that it satisfies condition~(iii) in Theorem~\ref{Thm:Kernel.lr} and has the following properties:
\begin{description}
\item[(a)] If $x \in \XX_o$, then $K(x,\cdot) = \delta_{S(x)}$, where $S : \XX_o \to \R$ is an arbitrary isotonic function such that $\Snw(x) \le S(x) \le \Sse(x)$ for all $x \in \XX_o$.
\item[(b)] If $x \in \XX \setminus \XX_o$, then $K(x,\cdot)$ has a bounded density $h(x,\cdot)$ with respect to $Q$ such that $h(x,y) = 0$ if $y < \Sse(x)$ or $y > \Snw(x)$. Moreover, $h$ is TP2 on the set $(\XX \setminus \XX_o) \times \R$, that means, \eqref{ineq:TP2.h} holds true for arbitrary numbers $x_1 < x_2$ in $\XX \setminus \XX_o$ and real numbers $y_1 < y_2$.
\end{description}
\end{Theorem}

\subsection{Weak convergence}
\label{Subsec:Weak.convergence}

Our starting point is the following basic result.

\begin{Lemma}
Consider a sequence $(R_n)_{n\ge 1}$ of distributions on $\R\times \R$ converging weakly to a distribution $R$. Suppose that each $R_n$ satisfies the conditions of Theorem~\ref{Thm:Kernel.st} or even of Theorem~\ref{Thm:Kernel.lr}. Then the limit $R$ satisfies the same conditions too.
\end{Lemma}

This lemma follows from condition~(ii) in both theorems, if we chose $D$ to be the set of all $y \in \R$ such that $P(\{y\}) = Q(\{y\}) = 0$. The question is whether and in what sense the corresponding stochastic kernel $K_n$ for $R_n$ described in condition~(iii) converges to the stochastic kernel $K$ for $R$.

Specifically, we consider conditional quantiles of $K_n(x,\cdot)$. For $\beta \in (0,1)$, let $q_n(\beta \,|\,x)$ be some real number $y$ such that
\[
	K_n(x,(-\infty,y)) \ \le \ \beta \ \le \ K_n(x,(-\infty,y]) .
\]
Since the kernel $K(x,\cdot)$ of the limiting distribution $R$ is not necessarily unique for all $x \in \XX$, we consider the extremal kernels $\Kw$ and $\Ke$ introduced in the proof of Theorem~\ref{Thm:TP2} and the corresponding minimal and maximal $\beta$-quantiles, respectively, that is,
\begin{align*}
	\qw(\beta \,|\, x) \
	&:= \ \min \bigl\{ y \in \R \colon \Kw(x,(-\infty,y]) \ge \beta \bigr\} , \\
	\qe(\beta \,|\, x) \
	&:= \ \max \bigl\{ y \in \R \colon \Ke(x,(-\infty,y)) \le \beta \bigr\} .
\end{align*}

\begin{Theorem}
\label{Thm:Weak.convergence}
Suppose that all $R_n$ (and thus $R$) satisfy the conditions of Theorem~\ref{Thm:Kernel.st}. For any $\beta \in (0,1)$ and arbitrary interior points $x_1 < x_2$ of $\XX$,
\begin{align*}
	\liminf_{n\to\infty} \, q_n(\beta \,|\, x_2) \
	&\ge \ \qw(\beta \,|\, x_1) , \\
	\limsup_{n\to\infty} \, q_n(\beta \,|\, x_1) \
	&\le \ \qe(\beta \,|\, x_2) .
\end{align*}
\end{Theorem}

\begin{Corollary}
\label{Cor:Weak.convergence}
Under the conditions of Theorem~\ref{Thm:Weak.convergence}, suppose that $\qw(\beta \,|\, x) = \qe(\beta \,|\, x)$ for some fixed $\beta \in (0,1)$ and all $x$ in an interval $(a,b) \subset \XX$. Then
\begin{align*}
	\lim_{n \to \infty} \, \sup_{x \in [a',b']} \bigl| q_n(\beta \,|\, x) - \qw(\beta \,|\, x) \bigr|
	\ = \ 0
\end{align*}
for arbitrary $a < a' < b' < b$.
\end{Corollary}

The previous findings lead naturally to consistent estimators of a TP2 distribution $R$ in principle. Suppose that $(\hat{R}_n)_n$ is a (random) sequence of distributions which converges to $R$ (almost surely) with respect to the bivariate Kuiper norm $\|\hat{R}_n - R\|_{\rm K}^{}$, say, where
\[
	\|\sigma\|_{\rm K}^{} \ := \ \sup_{a_1 < a_2, b_1 < b_2} \,
		\bigl| \sigma \bigl( (a_1,a_2] \times (b_1,b_2] \bigr) \bigr|
\]
for a finite signed measure $\sigma$ on $\R\times \R$. For instance, $\hat{R}_n$ could be the empirical distribution of independent random variables $(X_1,Y_1), \ldots, (X_n,Y_n)$ with distribution $R$. If $\check{R}_n$ is a TP2 distribution such that
\[
	\|\check{R}_n - \hat{R}_n\|_{\rm K}^{} \
	= \ \min_{\tilde{R} \ \text{a TP2 distribution}} \|\tilde{R} - \hat{R}_n\|_{\rm K}^{} ,
\]
then $\|\check{R}_n - R\|_{\rm K}^{}$ is bounded by $\|\check{R}_n - \hat{R}_n\|_{\rm K}^{} + \|R - \hat{R}_n\|_{\rm K}^{} \le 2  \|R - \hat{R}_n\|_{\rm K}^{}$. In particular, the sequence $(\check{R}_n)_n$ converges weakly to $R$ (almost surely). The existence of a TP2 approximation $\check{R}_n$ of $\hat{R}_n$ is not obvious for arbitrary distributions $\hat{R}_n$, but at least in case of $\hat{R}_n$ having finite support this is true.

\begin{Lemma}
\label{Lem:Minimum.Kuiper.distance}
Let $\hat{R}$ be a distribution with finite support. Then there exists a TP2 distribution $\check{R}$ minimising $\|\tilde{R} - \hat{R}\|_{\rm K}$ over all TP2 distributions $\tilde{R}$.
\end{Lemma}

The proof of this lemma reveals that the best approximating distribution $\check{R}$ may have finite support too, but it is not unique. Its explicit computation is yet an open problem.

\section{Proofs}
\label{Sec:Proofs}

\subsection{Isotonic densities}

For later purposes, we need a constructive version of the Radon-Nikodym theorem in a special case.

\begin{Lemma}
\label{Lem:IRN}
Let $\mu$ and $\nu$ be finite measures on $\R$ such that $\nu \le \mu$. Further, suppose that
\begin{equation}
\label{ineq:IRN1}
	\mu((y,z]) \nu((x,y]) \ \le \ \mu((x,y]) \nu((y,z])
	\quad\text{whenever} \ x < y < z .
\end{equation}
Then
\[
	f(x) \ := \ \sup_{a < x} \, \frac{\nu((a,x])}{\mu((a,x])} \ \in \ [0,1]
		\quad (\text{with} \ 0/0 := 0)
\]
defines an isotonic density of $\nu$ with respect to $\mu$. Moreover, if $\mu((x',x]) = 0$ for numbers $x' < x$, then $f(x') = f(x)$. Further, if $x \in \R$ satisfies $\mu((a,x]) > 0$ for all $a < x$, then
\[
	f(x) \ = \ \lim_{a \to x\,-} \, \frac{\nu((a,x])}{\mu((a,x])} .
\]
In particular, $f(x) = \nu(\{x\}) / \mu(\{x\})$ in case of $\mu(\{x\}) > 0$.
\end{Lemma}

\begin{Remark}
The density $f$ constructed in Lemma~\ref{Lem:IRN} is minimal in the sense that any isotonic density $\tilde{f}$ of $\nu$ with respect to $\mu$ satisfies $\tilde{f} \ge f$ pointwise. For if $f(x) > 0$, then for all $a < x$ with $\mu((a,x]) > 0$,
\[
	\mu((a,x])^{-1} \nu((a,x])
	\ = \ \mu((a,x])^{-1} \int_{(a,x]} \tilde{f}(w) \, \mu(\mathrm{d}w)
	\ \le \ \tilde{f}(x) .
\]
A maximal isotonic density could be constructed analogously: First note that with a simple approximation argument one can show that condition~\eqref{ineq:IRN1} is equivalent to the same condition with $[x_0,x_1)$ and $[x_1,x_2)$ in place of $(x_0,x_1]$ and $(x_1,x_2]$, respectively. Then one could define the density at $x$ by
\[
	\inf_{b > x} \, \frac{\nu([x,b))}{\mu([x,b))} \ \in \ [0,1]
		\quad (\text{with} \ 0/0 := 1) .
\]
\end{Remark}

\begin{proof}[\bf Proof of Lemma~\ref{Lem:IRN}]
Let $\gamma := \nu/\mu$ with $0/0 := 0$. Then $f(x) = \sup_{a < x} \gamma((a,x])$. Condition~\eqref{ineq:IRN1} is easily verified to be equivalent to
\begin{equation}
\label{ineq:IRN2}
	\gamma((x,y]) \ \le \ \gamma((y,z])
	\quad\text{whenever} \ x < y < z \ \text{and} \ \mu((x,z]) > 0 .
\end{equation}
As to isotonicity of $f$, let $x < y$. If $\mu((x,y]) = 0$, then $\nu((x,y]) = 0$, too, whence $\gamma((a,y]) = 0$ for $x \le a < y$ and $\gamma((a,y]) = \gamma((a,x])$ for $a < x$. Consequently,
\[
	f(y) \
	= \ \sup_{a < x} \, \gamma((a,y]) \
	= \ \sup_{a < x} \, \gamma((a,x]) \
	= \ f(x) .
\]
On the other hand, if $\mu((x,y]) > 0$, then \eqref{ineq:IRN2} implies that $\gamma((a,x]) \le \gamma((x,y]) \le f(y)$ for $a < x$, whence $f(x) \le \gamma((x,y]) \le f(y)$. This implies that
\begin{equation}
\label{ineq:IRN3}
	f(x) \mu((x,y]) \ \le \ \nu((x,y]) \ \le \ f(y) \mu((x,y])
	\quad\text{whenever} \ x < y .
\end{equation}

If $x \in \R$ is such that $\mu((a,x]) > 0$ for all $a < x$, then it follows from \eqref{ineq:IRN2} that for $a < a' < x$,
\[
	\gamma((a,x]) \
	= \ \frac{\mu((a,a'])}{\mu((a,x])} \, \gamma((a,a'])
		+ \frac{\mu((a',x])}{\mu((a,x])} \, \gamma((a',x])
	\ \le \ \gamma((a',x]) ,
\]
because $\gamma((a,a']) \le \gamma((a',x])$. Hence, $\gamma((a,x])$ is isotonic in $a < x$, and this implies the representation of $f(x)$ as $\lim_{a \to x\,-} \gamma((a,x])$. In particular,
\begin{equation}
\label{ineq:IRN3'}
	\nu(\{x\}) \ = \ f(x) \mu(\{x\})
	\quad\text{for all} \ x \in \R ,
\end{equation}
because this equation is trivial in case of $\mu(\{x\}) = 0$.

Now the previous inequalities \eqref{ineq:IRN3} and \eqref{ineq:IRN3'} are generalized as follows: For any bounded interval $I \subset \R$,
\begin{equation}
\label{ineq:IRN4}
	\inf_{z \in I} \, f(z) \, \mu(I) \
	\le \ \nu(I) \ \le \
	\sup_{z \in I} \, f(z) \mu(I) .
\end{equation}
In case of $I = (x,y]$, this follows from \eqref{ineq:IRN3} applied to $(x',y]$ with $x' \to x\,+$. In case of $I = (x,y)$, this follows from \eqref{ineq:IRN4} applied to $(x,y']$ with $y' \to y\,-$. In case of $I = [x,y]$ or $I = [x,y)$, we may deduce from \eqref{ineq:IRN3'} that
\[
	\nu(I) \
	= \ \nu(\{x\}) + \nu(I \setminus \{x\}) \
	= \ f(x) \mu(\{x\}) + \nu(I \setminus \{x\}) ,
\]
and then the assertion follows from applying the available inequalities to $I \setminus \{x\}$ instead of $I$.

It remains to be shown that $f$ is a density of $\nu$ with respect to $\mu$. That means, $\nu(I) = \int_I f \, d\mu$ for any bounded interval $I$. To this end we fix an arbitrary integer $k > 1$ and split $I$ into the disjoint intervals $I_1, \ldots, I_k$ where $I_1 := I \cap \{f \le 1/k\}$ and $I_j := I \cap \{(j-1)/k < f \le j/k\}$ for $2 \le j \le k$. Then it follows from \eqref{ineq:IRN4} that
\[
	\nu(I) \
	= \ \sum_{j=1}^k \nu(I_j) \
	\begin{cases}
		\displaystyle
		\le \ \sum_{j=1}^k (j/k) \mu(I_j) \
			\le \int_I (f + 1/k) \, \mathrm{d}\mu , \\[3ex]
		\displaystyle
		\ge \ \sum_{j=1}^k ((j-1)/k) \mu(I_j) \
			\ge \int_I (f - 1/k) \, \mathrm{d}\mu ,
	\end{cases}
\]
and letting $k \to \infty$ yields the asserted equation.
\end{proof}

\subsection{Proofs for Section~\ref{Sec:LROrder}}

We start with an elementary result about products and ratios of nonnegative numbers. The proof is elementary and thus omitted.

\begin{Lemma}
\label{Lem:Ratios.Products}
Let $r_1, r_2$ and $s_1, s_2$ be numbers in $[0,\infty)$ such that $(r_1,r_2), (s_1,s_2) \ne (0,0)$. Then
\[
	r_2 s_1 \ \le \ r_1 s_2
	\quad\text{if and only if}\quad
	\frac{r_2}{r_1} \ \le \ \frac{s_2}{s_1}
\]
(where the latter fractions are in $[0,\infty]$).
\end{Lemma}

\begin{proof}[\bf Proof of Theorem~\ref{Thm:EquivDef}]
We prove the result via a chain of implications, introducing two additional equivalent conditions:\\[1ex]
\textbf{(v)} \ There exists an isotonic densitiy $\rho : \R \to [0,1]$ of $Q_2$ with respect to $Q_1 + Q_2$.\\[0.5ex]
\textbf{(vi)} \ Suppose that $\mu$ is a $\sigma$-finite measure dominating $Q_1$ and $Q_2$. Then one can choose corresponding densities $g_j = \mathrm{d}Q_j/\mathrm{d}\mu$ ($j = 1,2$) such that $g_2/g_1$ is isotonic on the set $\{g_1 + g_2 > 0\}$.

\medskip

\noindent
\textbf{Step 1.} Note that the inequality in condition~(ii) is trivial whenever $g_1(x) = g_2(x) = 0$ or $g_1(y) = g_2(y) = 0$. Thus condition~(ii) is equivalent to the same condition with $x, y$ being restricted to the set $\{g_1 + g_2 > 0\} = \bigl\{ (g_1,g_2) \ne (0,0) \bigr\}$. But then it follows from Lemma~\ref{Lem:Ratios.Products} that conditions~(i) and (ii) are equivalent.

\noindent
\textbf{Step 2.} Suppose that condition~(ii) holds true. Then for Borel sets $A < B$,
\begin{align*}
	Q_1(B) Q_2(A) \
	&= \ \int 1_{A}(x) 1_{B}(y) g_1(y) g_2(x) \, \mu(\d x) \mu(\d y) \\
	&\le \ \int 1_{A}(x) 1_{B}(y) g_1(x) g_2(y) \, \mu(\d x) \mu(\d y) \\
	&= \ Q_1(A) Q_2(B) .
\end{align*}
Hence, condition~(iii) is satisfied as well, and obviously this implies condition~(iv), because the latter is a special case of the former.

\noindent
\textbf{Step 3.} Suppose that condition~(iv) is satisfied. With simple approximation argument one can show that condition~(iv) holds true with $D = \R$. Then we may apply Lemma~\ref{Lem:IRN} to $\mu := Q_1 + Q_2$ and $\nu := Q_2$. This yields an isotonic density $\rho : \R \to [0,1]$ of $Q_2$ with respect to $Q_1 + Q_2$. Hence, condition~(v) is satisfied as well.

\noindent
\textbf{Step 4.} Suppose that condition~(v) holds true. Let $\mu$ be a $\sigma$-finite measure dominating $Q_1, Q_2$. Then $\mu$ dominates $Q_1 + Q_2$, so by the Radon--Nikodym theorem there exists a density $h$ of $Q_1 + Q_2$ with respect to $\mu$. Consequently,
\begin{align*}
	Q_2(B) \
	&= \ \int_B \rho \, \d (Q_1 + Q_2)
		\ = \ \int_B \rho h \, \d \mu , \\
	Q_1(B) \
	&= \ \int_B (1 - \rho) \, \d (Q_1 + Q_2)
		\ = \ \int_B (1 - \rho) h \, \mathrm{d}\mu .
\end{align*}
Consequently, $g_2 := \rho h$ and $g_1 := (1 - \rho) h$ are densities of $Q_2$ and $Q_1$, respectively, with respect to $\mu$. On $\{g_1 + g_2 > 0\} = \{h > 0\}$, the ratio $g_2/g_1$ equals $\rho/(1 - \rho)$ and is isotonic. Hence, condition~(vi) is satisfied. Finally, since $\mu := Q_1 + Q_2$ is a finite measure dominating $Q_1, Q_2$, condition~(vi) implies condition~(i).
\end{proof}

\begin{proof}[\bf Proof of Lemma~\ref{Lem:LR_iff_condST}]
For notational convenience, we write $Q_{jC} := Q_j(\cdot \,|\, C)$ ($j = 1,2$). Suppose first that $Q_1 \lelr Q_2$ (property~(i)), and let $C$ be a Borel set with $Q_1(C),Q_2(C)>0$. Since $Q_1(B)Q_2(A)\le Q_1(A)Q_2(B)$ for all Borel sets $A,B$ such that $A < B$, and since $A\cap C < B\cap C$ too, we find that
\[
	Q_{1C}(B)Q_{2C}(A) = \frac{Q_1(B\cap C) Q_2(A\cap C)}{Q_1(C) Q_2(C)} \
	\le \ \frac{Q_1(A\cap C) Q_2(B\cap C)}{Q_1(C)Q_2(C)} = Q_{1C}(A)Q_{2C}(B) .
\]
Therefore, $Q_{1C} \lelr Q_{2C}$ (property~(ii)), and by Remark~\ref{Rem:ST_LR}, this implies that $Q_{1C} \lest Q_{2C}$ (property~(iii)).

Property~(iv) is obviously a consequence of property~(iii), so it remains to show that property~(iv) implies property~(i). Again, a simple approximation argument shows that property~(iv) holds true with $D = \R$. To verify that $Q_1 \lelr Q_2$, it suffices to show that
\begin{equation}
\label{ineq:Q12AB}
	Q_1(B) Q_2(A) \ \le \ Q_1(A)Q_2(B)
\end{equation}
for arbitrary $A = (x,y]$ and $B = (y,z]$ with $x < y < z$. With $C := A \cup B = (x,z]$, it suffices to consider the case $Q_1(C), Q_2(C) > 0$, because otherwise \eqref{ineq:Q12AB} is trivial. Then \eqref{ineq:Q12AB} is equivalent to
\[
	Q_{1C}(B) Q_{2C}(A) \ \le \ Q_{1C}(A) Q_{2C}(B) ,
\]
and since $Q_{jC}(A) = 1 - Q_{jC}(B)$, the latter inequality is equivalent to
\[
	Q_{1C}(B) \ \le \ Q_{2C}(B) .
\]
But this is a consequence of $Q_{1C} \lest Q_{2C}$, because $Q_{jC}(B) = Q_{jC}((y,\infty))$.
\end{proof}

\begin{proof}[\bf Proof of Lemma~\ref{Lem:LR_Partial_Order}]
Reflexivity of the likelihood ratio order is obvious. To show antisymmetry, note that $Q_1 \lelr Q_2$ and $Q_2 \lelr Q_1$ implies that $Q_1 \lest Q_2$ and $Q_2 \lest Q_1$. But the latter two inequalities mean that the distribution functions of $Q_1$ and $Q_2$ coincide, whence $Q_1 \equiv Q_2$.

It remains to prove transitivity. Suppose that $Q_1 \lelr Q_2 \lelr Q_3$. We want to show that the distributions $Q_1, Q_3$ satisfy condition~(iv) in Theorem~\ref{Thm:EquivDef}, that is,
\begin{equation}
\label{ineq:Q1Q3}
	Q_1((y,z]) Q_3((x,y]) \ \le \ Q_1((x,y]) Q_3((y,z])
\end{equation}
for arbitrary $x < y < z$. It suffices to consider the nontrivial situation that $Q_1((x,z])$ and $Q_3((x,z])$ are strictly positive. But then $Q_1 \lelr Q_2$ and $Q_2 \lelr Q_3$ implies that $Q_2((x,z]) > 0$ too. For if $Q_2((-\infty,x]) > 0 = Q_2((x,z])$, then
\[
	Q_1((x,z]) Q_2((-\infty,x]) \ \le \ Q_1((-\infty,x]) Q_2((x,z]) \ = \ 0 ,
\]
whence $Q_1((x,z]) = 0$. Likewise, if $Q_2((x,z]) = 0 < Q_2((z,\infty))$, then
\[
	Q_2((z,\infty)) Q_3((x,z]) \ \le \ Q_2((x,z]) Q_3((z,\infty)) \ = \ 0 ,
\]
whence $Q_3((x,z]) = 0$. Consequently, Lemma~\ref{Lem:Ratios.Products} implies that
\[
	\frac{Q_1((y,z])}{Q_1((x,y])} \
	\le \ \frac{Q_2((y,z])}{Q_2((x,y])} \
	\le \ \frac{Q_3((y,z])}{Q_3((x,y])} ,
\]
and this implies inequality~\eqref{ineq:Q1Q3}.
\end{proof}

\begin{proof}[\bf Proof of Corollary~\ref{Cor:LR_ROC}]
Suppose first that $Q_1 \lelr Q_2$. Let $(a_1,a_2)$, $(b_1,b_2)$ and $(c_1,c_2)$ be three different points in $\mathrm{ROC}(P,Q)$ such that $a_1 \le b_1 \le c_1$ and $a_2 \le b_2 \le c_2$. For $\ell = a,b,c$, we may write $(\ell_1,\ell_2) = \bigl( Q_1(H_\ell), Q_2(H_\ell) \bigr)$ with a halfline $H_\ell \in \mathcal{H}$, where $H_a \subsetneq H_b \subsetneq H_c$. But then $C := H_c \setminus H_b < H_b\setminus H_a =: B$ element-wise, and condition~(iii) in Theorem~\ref{Thm:EquivDef} implies that
\[
	(b_1 - a_1)(c_2 - b_2) \ = \ Q_1(B) Q_2(C)
	\ \le \ Q_1(C) Q_2(B) \ = \ (c_1 - b_1)(b_2 - a_2) .
\]
Since $(r_1,r_2) := (b_1 - a_1, b_2 - a_2)$ and $(s_1,s_2) := (c_1 - b_1, c_2 - b_2)$ differ from $(0,0)$ by assumption, Lemma~\ref{Lem:Ratios.Products} shows that the latter displayed inequality is equivalent to
\[
	\frac{b_2 - a_2}{b_1 - a_1} \ \ge \ \frac{c_2 - b_2}{c_1 - b_1} .
\]
This shows that the ROC curve of $P$ and $Q$ is concave.

Now suppose that $\mathrm{ROC}(P,Q)$ is concave. This implies that for arbitrary points $(a_1,a_2)$, $(b_1,b_2)$, $(c_1,c_2)$ in $\mathrm{ROC}(P,Q)$ with $a_1 \le b_1 \le c_1$ and $a_2 \le b_2 \le c_2$,
\[
	(b_1 - a_1)(c_2 - b_2) \ 
	\le \ (c_1 - b_1)(b_2 - a_2) ,
\]
see Lemma~\ref{Lem:Ratios.Products}. Specifically, let $(\ell_1,\ell_2) = \bigl( Q_1(H_\ell), Q_2(H_\ell) \bigr)$ for $\ell = a,b,c$ with halflines $H_c := (x,\infty)$, $H_b := (y,\infty)$ and $H_a := (z,\infty)$, where $x < y < z$. Then the latter displayed inequality reads
\[
	Q_1((y,z]) Q_2((x,y]) \ \le \ Q_1((x,y]) Q_2((y,z]) ,
\]
whence $Q_1$ and $Q_2$ satisfy condition~(iv) of Theorem~\ref{Thm:EquivDef}.
\end{proof}

The proof of Corollary~\ref{Cor:Equiv_LR_ODC} uses elementary inequalities for distribution and quantile functions; see for instance Chapter~1 of \cite{Shorack_Wellner_1986}:

\begin{Lemma}
\label{Lem:Quant}
For any distribution function $G$ and all $\alpha\in [0,1]$, we have $G(G^{-1}(\alpha)) \ge \alpha$, with equality if, and only if, $\alpha \in G(\Rbar)$. For all $y \in \R$, $G^{-1}(G(y)) \le y$, with equality if, and only if, $G(y_o)<G(y)$ for all $y_o < y$.
\end{Lemma}

\begin{proof}[\bf Proof of Corollary~\ref{Cor:Equiv_LR_ODC}]
Let us write $H$ instead of $H_{G_1,G_2}$ to lighten the notation. Suppose that $Q_1 \lelr Q_2$, and let $r,s,t\in G_1(\Rbar)$ with $r<s<t$. Lemma~\ref{Lem:Quant} yields that $G_1^{-1}(r) < G_1^{-1}(s) < G_1^{-1}(t)$, so $A := \bigl( G_1^{-1}(r),G_1^{-1}(s) \bigr]$ and $B := \bigl( G_1^{-1}(s),G_1^{-1}(t) \bigr]$ are such that
\begin{align*}
	Q_1(A) \
	&= \ G_1^{}(G_1^{-1}(s)) - G_1^{}(G_1^{-1}(r)) \ = \ s-r \ > \ 0, \\
	Q_1(B) \
	&= \ G_1^{}(G_1^{-1}(t)) - G_1^{}(G_1^{-1}(s)) \ = \ t-s \ > \ 0.
\end{align*}
Thus,
\[
	(t - s) \bigl( H(s) - H(r) \bigr)
	\ = \ Q_1(B) Q_2(A) \ \le \ Q_1(A) Q_2(B)
	\ = \ (s - r) \bigl( H(t) - H(s) \bigr) ,
\]
and dividing both sides by $(s-r)(t-s)$ shows that $H$ is convex on $G_1(\Rbar)$.

Suppose now that $H$ is convex on $G_1(\Rbar)$. To verify that $Q_1 \lelr Q_2$, we have to show that
\begin{equation}
\label{ineq:Q12xyz}
	Q_1((x,y]) Q_2((y,z]) - Q_1((y,z]) Q_2((x,y]) \ \ge \ 0
\end{equation}
for arbitrary real numbers $x < y < z$. Since $Q_2 \ll Q_1$, inequality~\eqref{ineq:Q12xyz} is trivial if $Q_1((x,y]) = 0$ or $Q_1((y,z]) = 0$. Hence, it suffices to verify \eqref{ineq:Q12xyz} in case of $Q_1((x,y]) = G_1(y) - G_1(x) > 0$ and $Q_1((y,z]) = G_1(z) - G_1(y) > 0$. For $w = x,y,z$, let $\tilde{w} := G_1^{-1}(G_1^{}(w)) \le w$, and note that $G_1(\tilde{w}) = G_1(w)$, so $Q_1((\tilde{w},w]) = 0 = Q_2((\tilde{w},w])$. Consequently, with $r := G_1(x)$, $s := G_1(y)$ and $t := G_1(z)$, the left-hand side of \eqref{ineq:Q12xyz} equals
\begin{align*}
	(s - r) Q_2((\tilde{y},\tilde{z}])
		- (t - s) Q_2((\tilde{x},\tilde{y}]) \
	&= \ (s - r) \bigl( H(t) - H(s) \bigr)
		- (t - s) \bigl( H(s) - H(r) \bigr) \\
	&= \ (s - r)(t - s)
		\Bigl( \frac{H(t) - H(s)}{t - s}
			- \frac{H(s) - H(r)}{s - r} \Bigr) \\
	&\ge \ 0
\end{align*}
by convexity of $H$.
\end{proof}

\subsection{Proofs for Section~\ref{Sec:Bivariate}}

\begin{proof}[\bf Proof of Theorem~\ref{Thm:Kernel.st}]
Condition~(iii) states that $K(x,\cdot)$ is isotonic in $x \in \XX$ with respect to stochastic order. That means, for any fixed number $y$, $K(x,(y,\infty))$ is isotonic in $x \in \XX$. Consequently, for Borel sets $A_1, A_2$ with $A_1 < A_2$,
\begin{align*}
	R(A_1 \times (y,\infty)) P(A_2) \
	&= \ \int_\XX 1_{A_1}(x) K(x,(y,\infty)) \, P(\d x) \int_{\XX} 1_{A_2}(x) P(\d x) \\
	&\le \ P(A_1) K(x_o, (y,\infty)) P(A_2) \\
	&\le \ \int_{\XX} 1_{A_1}(x) \, P(\d x) \int_{\XX} 1_{A_2}(x) K(x,(y,\infty)) \, P(\d x) \\
	&= \ P(A_1) R(A_2 \times (y,\infty)) ,
\end{align*}
where $x_o$ is an arbitrary number in $\XX$ such that $A_1\cap \XX \le \{x_o\} \le A_2\cap \XX$. This shows that condition~(iii) implies condition~(i).

Condition~(ii) is just a special case of condition~(i), so it remains to deduce condition~(iii) from condition~(ii). Note first that via a simple approximation argument, condition~(ii) extends to $D = \R$. Now we apply Lemma~\ref{Lem:IRN} to $\mu(\cdot) := P$ and $\nu(\cdot) := \Pr(X \in\cdot, Y > y)$ with arbitrary $y \in \R$. This yields
\[
	S(x,y) \ := \ \sup_{a < x} \, \Pr(Y > y \,|\, a < X \le x)
\]
with $\Pr(\cdot \,|\, a < X \le x) := 0$ in case of $\Pr(a < X \le x) = 0$. According to Lemma~\ref{Lem:IRN}, $S(\cdot, y)$ is an isotonic density of $\nu$ with respect to $\mu$. Moreover, $S(x,\cdot)$ is antitonic and right-cont\-inu\-ous. Antitonicity is obvious. To verify right-continuity, note that $S(x,y) = 0$ implies that $S(x,\cdot) \equiv 0$ on $[y,\infty)$. If $S(x,y) > 0$,
\[
	\liminf_{z \to y\,+} \, S(x,z) \
	\ge \ \liminf_{z \to y\,+} \, \Pr(Y > z \,|\, a < X \le x) \
	= \ \Pr(Y > y \,|\, a< X \le x)
\]
for any fixed $a < x$ such that $\Pr(a < X \le x) > 0$. This shows that $\liminf_{z \to y\,+} S(x,z) \ge S(x,y)$. Now the idea is to interpret $S(x,y)$ as $K(x,(y,\infty))$ for some probability measure $K(x,\cdot)$ whenever $x \in \XX$. Such a probability measure $K(x,\cdot)$ exists if and only if
\[
	\lim_{y \to -\infty} \, S(x,y) \ = \ 1
	\quad\text{and}\quad
	\lim_{y \to \infty} \, S(x,y) \ = \ 0 .
\]
Indeed, if $\Pr(X \le x) > 0$, then there exists a real number $a < x$ such that $\Pr(a < X \le x) > 0$. Hence, by definition of $S(x,y)$,
\[
	S(x,y) \ \ge \ \Pr(Y > y \,|\, a < X \le x) \ \to \ 1
	\quad\text{as} \ y \to - \infty .
\]
Similary, if $\Pr(X \ge x) > 0$, then $\Pr(X = x) > 0$ or $\Pr(x < X \le x') > 0$ for some real number $x' > x$. In the former case, Lemma~\ref{Lem:IRN} entails that
\[
	S(x,y) \ = \ \Pr(Y > y \,|\, X = x) \ \to \ 0
	\quad\text{as} \ y \to \infty .
\]
In the latter case, it follows from condition~(ii) that
\[
	S(x,y) \ \le \ \Pr(Y > y \,|\, x < X \le x') \ \to \ 0
	\quad\text{as} \ y \to \infty .
\]
These considerations show that for any $x\in\R$,
\[
	K(x,(y,\infty)) \ := \ \begin{cases}
		S(x,y) & \text{if} \ x \in \XX , \\
		Q((y,\infty)) & \text{if} \ x \not\in \XX ,
	\end{cases}
\]
defines a probability measure $K(x,\cdot)$ on $\R$ such that $K(x_1, \cdot) \lest K(x_2,\cdot)$ for $x_1, x_2 \in \XX$ with $x_1 < x_2$. Recall that $S(\cdot,y)$ is a Radon-Nikodym derivative of $\Pr(X \in \cdot, Y > y)$ with respect to $P$. Consequently, $K(\cdot,B)$ is measurable and satisfies the equation $\Pr(X \in A, Y \in B) = \Ex \bigl( 1_A(X) K(X,B) \bigr)$ for all sets $A \in \BB$ and every open halfline $B = (y,\infty)$, $y \in \R$. Since the latter family of halflines is closed under intersections and generates $\BB$, the preceding properties of $K(\cdot,B)$ extend to any set $B \in \BB$.
\end{proof}

\begin{proof}[\bf Proof of Theorem~\ref{Thm:Kernel.lr}]
Condition~(iii) states that $K(x,\cdot)$ is isotonic in $x \in \XX$ with respect to likelihood ratio order. This implies that for Borel sets $A_1,A_2,B_1,B_2$ with $A_1 < A_2$ and $B_1 < B_2$ and for independent copies $X', X''$ of $X$,
\begin{align*}
	\Pr(&X \in A_1, Y \in B_2) \Pr(X \in A_2, Y \in B_1) \\
	&= \ \Ex \bigl( 1_{A_1 \cap \XX}(X') 1_{A_2\cap \XX}(X'')
		K(X',B_2) K(X'',B_1) \bigr) \\
	&\le \ \Ex \bigl( 1_{A_1 \cap \XX}(X') 1_{A_2\cap \XX}(X'')
		K(X',B_1) K(X'',B_2) \bigr) \\
	&= \ \Pr(X \in A_1, Y \in B_1) \Pr(X \in A_2, Y \in B_2) .
\end{align*}
Hence, condition~(iii) implies condition~(i). Since condition~(ii) is a special case of condition~(i), it remains to show that condition~(ii) implies condition~(iii). One can easily show that condition~(ii) in Theorem~\ref{Thm:Kernel.lr} implies condition~(ii) in Theorem~\ref{Thm:Kernel.st} by letting $y_0 \to -\infty$ and $y_2 \to \infty$. Thus we may construct the stochastic kernel $K$ precisely as in the proof of Theorem~\ref{Thm:Kernel.st}.

It remains to be shown that $K(x_1,\cdot) \lelr K(x_2,\cdot)$ for $x_1, x_2 \in \XX$ with $x_1 < x_2$. The construction of $K$ implies the following facts: If $\Pr(x_1 < X \le x_2) = 0$, then $K(x_1,\cdot) \equiv K(x_2,\cdot)$, and the assertion is trivial. Otherwise, for $j = 1,2$, \[
	x_j' \ := \ \min \bigl( \{x_j\} \cup \{x < x_j : \Pr(x < X \le x_j) = 0\} \bigr)
\]
is well-defined, and $x_1' \le x_1 < x_2' \le x_2$. Moreover, $\Pr(a < X \le x_j') > 0$ for all $a < x_j'$, whence Lemma~\ref{Lem:IRN} implies that
\[
	K(x_j,(y,\infty)) \ = \ \lim_{a \to x_j'\,-} \Pr(Y > y \,|\, a < X \le x_j)
\]
for all $y \in \R$. Consequently, for $y_0 < y_1 < y_2$, the difference
\[
	K(x_1,(y_0,y_1]) K(x_2,(y_1,y_2]) - K(x_1,(y_1,y_2]) K(x_2,(y_0,y_1])
\]
is the limit of the difference
\begin{align*}
	\Pr(&y_0 < Y \le y_1 \,|\, a < X \le x_1)
			\Pr(y_1 < Y \le y_2 \,|\, b < X \le x_2) \\
	&- \ \Pr(y_1 < Y \le y_2 \,|\, a < X \le x_1)
			\Pr(y_0 < Y \le y_1 \,|\, b < X \le x_2)
\end{align*}
as $a \to x_1'\,-$ and $b \to x_2'\,-$. But the latter difference is nonnegative as soon as $x_1 \le b$. In case of $\Pr(x_1 < X \le b, y_0 < Y \le y_2) = 0$, this follows from condition~(ii) applied with $(a, x_1, x_2)$ in place of $(x_0,x_1,x_2)$, noting that $\Pr(\cdot \,|\, b < X \le x_2) \equiv \Pr(\cdot \,|\, x_1 < X \le x_2)$. In case of $\Pr(x_1 < X \le b, y_0 < Y \le y_2) > 0$, one has to apply condition~(ii) with $(a, x_1, b)$ and then with $(x_1, b, x_2)$ in place of $(x_0,x_1,x_2)$ and apply Lemma~\ref{Lem:Ratios.Products} twice.
\end{proof}

In the subsequent proof, we reinterpret occasionally a distribution on $\R$ as a distribution on $[-\infty,\infty]$ with mass zero on $\{-\infty,\infty\}$.

\begin{proof}[\bf Proof of Theorem~\ref{Thm:TP2}]
We start with a general observation which applies to any kernel $K$ as in part~(iii) of Theorem~\ref{Thm:Kernel.lr}. Let $x \in \XX$ and let $A_1, A_2 \in \BB$ such that $\Pr(X \in A_1), \Pr(X \in A_2) > 0$ and $A_1 \le \{x\} \le A_2$. Then,
\begin{equation}
\label{eq:3Qs}
	\Pr(Y > \Snw(x) \,|\, X \in A_1) \
	= \ 0 \
	= \ \Pr(Y < \Sse(x) \,|\, X \in A_2)
\end{equation}
and
\begin{equation}
\label{ineq:3Qs}
	\LL(Y \,|\, X \in A_1) \
	\lelr \ K(x,\cdot) \
	\lelr \ \LL(Y \,|\, X \in A_2) .
\end{equation}
Equality~\eqref{eq:3Qs} follows immediately from the definition of $\Snw(x)$ and $\Sse(x)$. To verify inequality~\eqref{ineq:3Qs}, consider sets $B_1 < B_2$ in $\BB$. Then the order constraint about $K$ implies that
\begin{align*}
	\Pr(Y \in B_2 \,|\, X \in A_1) K(x,B_1) \
	&= \ \Pr(X \in A_1)^{-1} \Ex \bigl( 1_{A_1 \cap \XX}(X) K(X,B_2) K(x,B_1) \bigr) \\
	&\le \ \Pr(X \in A_1)^{-1} \Ex \bigl( 1_{A_1 \cap \XX}(X) K(X,B_1) K(x,B_2) \bigr) \\
	&= \ \Pr(Y \in B_1 \,|\, X \in A_1) K(x,B_2) ,
\end{align*}
and the same reasoning yields the inequality
\[
	K(x,B_2) \Pr(Y \in B_1 \,|\, X \in A_2) \
	\le \ K(x,B_1) \Pr(Y \in B_2 \,|\, X \in A_2) .
\]

Now let $\Kw$ be the specific kernel constructed in the proof of Theorems~\ref{Thm:Kernel.st} and \ref{Thm:Kernel.lr}. For $x \in \XX$ and $y \in \R$, $\Kw(x,(y,\infty))$ is the supremum of $\Pr(Y > y \,|\, a < X \le x)$ over all $a < x$ such that $\Pr(a < X \le x) > 0$. This implies that
\begin{equation}
\label{ineq:Kw.Sw}
	\Kw(x,(y,\infty)) \ \begin{cases}
		= 0 & \text{if} \ y \ge \Snw(x) , \\
		> 0 & \text{if} \ y < \Snw(x) ,
	\end{cases}
\end{equation}
because $\Pr(X \le x, Y > y) = 0$ implies that $\Pr(a < X \le x, Y > y) = 0$ for all $a < x$, while $\Pr(X \le x, Y > y) > 0$ implies that $\Pr(a < X \le x, Y > y) > 0$ for some $a < x$. Analogously one could construct a kernel $\Ke$: For $x \in \XX$ let $\Ke(x,(-\infty,y))$ be the supremum of $\Pr(Y < y \,|\, x \le X < b)$ over all $b > x$ such that $\Pr(x \le X < b) > 0$, and for $x \in \R \setminus \XX$ let $\Ke(x,\cdot) := Q$. Then we have a second version $\Ke$ of $K$ which satisfies condition~(iii) in Theorem~\ref{Thm:Kernel.lr}, and now for $x \in \XX$ and $y \in \R$,
\begin{equation}
\label{ineq:Ke.Se}
	\Ke(x,(-\infty,y)) \ \begin{cases}
		= 0 & \text{if} \ y \le \Sse(x) , \\
		> 0 & \text{if} \ y > \Sse(x) .
	\end{cases}
\end{equation}
A first consequence of these constructions is that
\begin{equation}
\label{eq:bracketing.Snw.Sse}
	\Pr \bigl( \Sse(X) \le Y \le \Snw(X) \bigr) \ = \ 1 .
\end{equation}
The two kernels $\Kw, \Ke$ are extremal in the sense that for any kernel $K$ satisfying condition~(iii) in Theorem~\ref{Thm:Kernel.lr} and all $x \in \XX$,
\begin{equation}
\label{ineq:bracketing.Kw.K.Ke}
	\Kw(x,\cdot) \ \lelr \ K(x, \cdot) \ \lelr \ \Ke(x, \cdot) .
\end{equation}
This can be deduced from Remark~\ref{Rem:Weak.convergence1} and inequalities~\eqref{ineq:3Qs} as follows: With $x_{\rm w} := \min\{x' \le x : \Pr(x' < X \le x) = 0\}$, the distribution $\Kw(x, \cdot)$ is the weak limit of
\[
	\LL(Y \,|\, a < X \le x) \ \lelr \ K(x,\cdot)
	\quad\text{as} \ a \to x_{\rm w}- ,
\]
and with $x_{\rm e} := \max\{x' \ge x : \Pr(x \le X < x') = 0\}$, the distribution $\Ke(x, \cdot)$ is the weak limit of
\[
	\LL(Y \,|\, x \le X < b) \ \gelr \ K(x,\cdot)
	\quad\text{as} \ b \to x_{\rm e}+ .
\]
An important consequence of inequalities \eqref{ineq:Kw.Sw}, \eqref{ineq:Ke.Se} and \eqref{ineq:bracketing.Kw.K.Ke} is that
\begin{equation}
\label{ineq:modified.K}
	K \bigl( x, \bigl[ \Sse(x), \Snw(x) \bigr] \bigr) \ > \ 0
	\quad\text{for} \ x \in \XX \setminus \XX_o .
\end{equation}
Indeed, if $K \bigl( x, \bigl( -\infty,\Sse(x) \bigr) \bigr) > 0$, then the inequalities \eqref{ineq:Kw.Sw} and $\Kw(x, \cdot) \lelr K(x, \cdot)$ imply that
\begin{align*}
	0 \
	&< \ \Kw \bigl( x, \bigl[ \Sse(x), \Snw(x) \bigr] \bigr)
		K \bigl( x, \bigl( -\infty, \Sse(x) \bigr) \bigr) \\
	&\le \ \Kw \bigl( x, \bigl( -\infty, \Sse(x) \bigr) \bigr)
		K \bigl( x, \bigl[ \Sse(x), \Snw(x) \bigr] \bigr)
		\ \le \ K \bigl( x, \bigl[ \Sse(x), \Snw(x) \bigr] \bigr) .
\end{align*}
Likewise, if $K \bigl( x, \bigl( \Snw(x), \infty \bigr) \bigr) > 0$, then the inequalities \eqref{ineq:Ke.Se} and $K(x,\cdot) \lelr \Ke(x, \cdot)$ imply that
\begin{align*}
	0 \
	&< \ K \bigl( x, \bigl( \Snw(x), \infty \bigr) \bigr)
		\Ke \bigl( x, \bigl[ \Sse(x), \Snw(x) \bigr] \bigr) \\
	&\le \ K \bigl( x, \bigl[ \Sse(x), \Snw(x) \bigr] \bigr)
		\Ke \bigl( x, \bigl( \Snw(x), \infty \bigr) \bigr)
		\ \le \ K \bigl( x, \bigl[ \Sse(x), \Snw(x) \bigr] \bigr) .
\end{align*}
Consequently, if we chose an arbitrary isotonic function $S : \XX_o \to \R$ such that $\Snw \le S \le \Sse$ on $\XX_o$, we may replace $K$ with $K_{\rm new}$ given by
\[
	K_{\rm new}(x, B) \ := \ \begin{cases}
		K \bigl( x, \bigl[ \Sse(x), \Snw(x) \bigr] \bigr)^{-1}
			K \bigl( x, B \cap \bigl[ \Sse(x), \Snw(x) \bigr] \bigr)
			& \text{if} \ x \in \XX \setminus \XX_o , \\
		\delta_{S(x)}(B)
			& \text{if} \ x \in \XX_o , \\
		K(x, B)
			& \text{if} \ x \in \R \setminus \XX .
	\end{cases}
\]
Indeed, it follows from \eqref{eq:bracketing.Snw.Sse} that $K \bigl( X, \bigl[ \Sse(X), \Snw(X) \bigr] \bigr) = 1$ almost surely, so $K_{\rm new}(X, \cdot) \equiv K(X, \cdot)$ almost surely. Consequently, $K_{\rm new}$ describes the conditional distribution of $Y$ given $X$, too. Moreover, it satisfies condition~(iii) of Theorem~\ref{Thm:Kernel.lr}: Let $x_1, x_2 \in \XX$ with $x_1 < x_2$. If $x_1$ and $x_2$ lie in $\XX \setminus \XX_o$, it follows from $K(x_1, \cdot) \lelr K(x_2,\cdot)$ and part~(i) of the subsequent Lemma~\ref{Lem:Truncation.lr} that $K_{\rm new}(x_1,\cdot) \lelr K_{\rm new}(x_2,\cdot)$, too. If at least one of the two points $x_1, x_2$ lies in $\XX_o$, then $K_{\rm new}(x_1, (-\infty,y]) = 1 = K_{\rm new}(x_2, [y,\infty))$ for some $y \in \R$, so $K_{\rm new}(x_1, \cdot) \lelr K_{\rm new}(x_2,\cdot)$ by Lemma~\ref{Lem:Truncation.lr}~(ii). Precisely, if $x_1 \in \XX_o$ and $x_2 \in \XX \setminus \XX_o$, then $S(x_1) \le \Sse(x_1) \le \Sse(x_2)$, so the asserted inequalities hold for any $y$ between $S(x_1)$ and $\Sse(x_2)$. Likewise, if $x_1 \in \XX \setminus \XX_o$ and $x_2 \in \XX_o$, we can pick any $y$ between $\Snw(x_1)$ and $S(x_2)$. Finally, if $x_1, x_2 \in \XX_o$, any $y$ between $S(x_1)$ and $S(x_2)$ is suitable.

Finally, let $x \in \XX \setminus \XX_o$. Then we may apply Lemma~\ref{Lem:LR.sandwich} to $Q_0 := \LL(Y \,|\, X \le x)$, $Q_* := K_{\rm new}(x,\cdot)$, $Q_1 := \LL(Y \,|\, X > x)$ and $\lambda_0 := \Pr(X \le x)$, $\lambda_1 := \Pr(X > x)$, so $\lambda_0 Q_0 + \lambda_1 Q_1 = Q = \LL(Y)$. In case of $\Pr(X > x) = 0$, we replace ``$\le x$'' with ``$< x$'' and ``$> x$'' with ``$\ge x$''. Here we exclude the trivial situation that $\Pr(X = x) = 1$. This shows that
\[
	h(x,y) \ := \ \lim_{w \to y\,-} \, \frac{K_{\rm new}(x,(w,y])}{Q((w,y])}
\]
(with $0/0 := 0$) defines a bounded density of $K_{\rm new}(x,\cdot)$ with respect to $Q$, and obviously, $h(x,y) = 0$ if $y < \Sse(x)$ or $y > \Snw(x)$. The explicit representation of $h(x,y)$ in terms of probability ratios $K(x,(w,y])/Q((w,y])$, $w < y$, and the fact that $K_{\rm new}(x_1,\cdot) \lelr K_{\rm new}(x_2,\cdot)$ for arbitrary points $x_1 < x_2$ in $\XX$ implies that $h$ is TP2 on $\XX \setminus \XX_o \times \R$.
\end{proof}

\begin{Lemma}
\label{Lem:Truncation.lr}
Let $Q_1, Q_2$ be probability distributions on $\R$.
\begin{description}
\item[(i)] If $[a_1, b_1]$ and $[a_2, b_2]$ are intervals in $[-\infty,\infty]$ such that $a_1 \le a_2$, $b_1 \le b_2$, and $Q_j([a_j,b_j]) > 0$ for $j = 1,2$, then it follows from $Q_1 \lelr Q_2$ that $Q_1(\cdot \,|\, [a_1,b_1]) \lelr Q_2(\cdot \,|\, [a_2,b_2])$, too.

\item[(ii)] If $Q_1((-\infty,y]) = 1 = Q_2([y,\infty))$ for some $y \in \R$, then $Q_1 \lelr Q_2$.
\end{description}
\end{Lemma}

\begin{proof}[\bf Proof of Lemma~\ref{Lem:Truncation.lr}]
As to part~(i), $Q_1 \lelr Q_2$ implies that $Q_1(\cdot \,|\, C) \lelr Q_2(\cdot \,|\, C)$ for any Borel set $C$ with $Q_1(C), Q_2(C) > 0$, see Lemma~\ref{Lem:LR_iff_condST}. Applying this fact with $C = [a_1,b_2]$ shows that it suffices to prove part~(i) with $a_1 = -\infty$ and $b_2 = \infty$. But then, by transitivity of $\lelr$, it suffices to show that $Q_1(\cdot \,|\, (-\infty,b_1]) \lelr Q_1$ and $Q_2 \lelr Q_2(\cdot \,|\, [a_2,\infty))$. These assertions follow from observing that $Q_1(\cdot \,|\, (-\infty,b_1])$ has antitonic density $Q_1((-\infty,b_1])^{-1} 1_{(-\infty,b_1]}$ with respect to $Q_1$, and $Q_2(\cdot \,|\, [a_2,\infty))$ has isotonic density $Q_2([a_2,\infty))^{-1} 1_{[a_2,\infty)}$ with respect to $Q_2$.

As to part~(ii), a density $\rho$ of $Q_2$ with respect to $Q_1 + Q_2$ is given by $\rho(x) = 1_{[x = y]} \gamma + 1_{[x > y]}$, where $\gamma := Q_2(\{y\}) / (Q_1 + Q_2)(\{y\}) \in [0,1]$. Thus $\rho$ is isotonic, so $Q_1 \lelr Q_2$.
\end{proof}

\begin{Lemma}
\label{Lem:Left.support.Q}
Let $Q$ be a distribution on $\R$, and let $\supp_{\rm left}(Q)$ be the set of all $y \in \R$ such that $Q((x,y]) > 0$ for arbitrary $x < y$. Then $Q(\supp_{\rm left}(Q)) = 1$.
\end{Lemma}

\begin{proof}[\bf Proof of Lemma~\ref{Lem:Left.support.Q}]
If $Q(I) > 0$ for any nonvoid open interval, then $\supp_{\rm left}(Q) = \R$, so the claim is obvious. Otherwise, let $\mathcal{J}$ be the family of maximal open intervals $J \subset \R$ such that $Q(J) = 0$. This family $\mathcal{J}$ is finite or countable. A point $y$ belongs to $\R \setminus \supp_{\rm left}(Q)$ if and only if for some $J \in \mathcal{J}$, either $y \in J$, or $y = \sup(J)$ and $Q(\{y\}) = 0$. Consequently, $\R \setminus \supp_{\rm left}(Q)$ is equal to the union of $\bigcup_{J \in \mathcal{J}} J$ and a finite or countable set of points $y$ such that $Q(\{y\}) = 0$. This shows that $Q(\R \setminus \supp_{\rm left}(Q)) = 0$.
\end{proof}

\begin{Lemma}
\label{Lem:LR.sandwich}
Let $Q_0, Q_*, Q_1$ be probability distributions on $\R$ such that $Q_0 \lelr Q_* \lelr Q_1$ and $Q_0((x,\infty)) + Q_1((-\infty,x)) > 0$ for arbitrary $x \in \R$. Let $Q := \lambda_0 Q_0 + \lambda_1 Q_1$ with $\lambda_0, \lambda_1 > 0$ such that $\lambda_0 + \lambda_1 = 1$. Then
\[
	g(y) \ := \ \lim_{x \to y\,-} \, \frac{Q_*((x,y])}{Q((x,y])}
\]
with $0/0 := 0$ exists and defines a bounded density of $Q_*$ with respect to $Q$.
\end{Lemma}

\begin{proof}[\bf Proof of Lemma~\ref{Lem:LR.sandwich}]
By assumption, for $j = 0,1$, there exists a density $\rho_j$ of $Q_*$ with respect to $Q_j + Q_*$ with values in $[0,1]$, where $\rho_0$ is isotonic and $\rho_1$ is antitonic. Note that $1 - \rho_j$ is automatically a version of $dQ_j/d(Q_j + Q_*)$. An important fact is that $\rho_0 + \rho_1$ is bounded away from $2$, i.e.\ $\rho_0 + \rho_1 \le 2 - \delta$ for some $\delta > 0$. To verify this, suppose the contrary. Then there exists a sequence $(x_k)_k$ such that $\lim_{k \to \infty} \rho_j(x_k) = 1$ for $j = 0,1$. Without loss of generality let $(x_k)_k$ converge to a point $x \in [-\infty,\infty]$. If $x < \infty$, then $\rho_0 \equiv 1$ on $(x,\infty)$, whence $Q_0((x,\infty)) = 0$. Likewise, if $x > - \infty$, then $\rho_1 \equiv 1$ on $(-\infty,x)$, whence $Q_1((-\infty,x)) = 0$. In particular, $x$ has to be a real number, but then our assumption on $Q_0$ and $Q_1$ is obviously violated.

Note first that $A_0 := \{\rho_0 = 0\} < \{\rho_0 > 0\}$ satisfies $Q_*(A_0) = 0$, and then $Q_* \lelr Q_1$ implies that $Q_1(A_0) = 0$. Similarly, $A_1 := \{\rho_1 = 0\} > \{\rho_1 > 0\}$ satisfies $Q_*(A_1) = 0 = Q_0(A_1)$. Thus the real line may be partitioned into the sets $A_0 < A_* := \{\rho_0\rho_1 > 0\} < A_1$, and $Q_*(A_*) = 1$. For measurable functions $h \ge 0$ on $A_*$ and $j = 0,1$,
\[
	\int_{A_*} h \, dQ_* \ = \ \int h \rho_j \, d(Q_* + Q_j) ,
\]
so
\[
	\int_{A_*} h (1 - \rho_j) \, dQ_* \ = \ \int h \rho_j \, dQ_j ,
\]
and replacing $h$ with $h/\rho_j$ shows that
\[
	\int_{A_*} h (\rho_j^{-1} - 1) \, dQ_* \ = \ \int h \, dQ_j .
\]
Consequently,
\[
	\int_{A_*} h
		\bigl( \lambda_0 (\rho_0^{-1} - 1) + \lambda_1 (\rho_1^{-1} - 1) \bigr)
			\, dQ_*
	\ = \ \int_{A_*} h \, dQ .
\]
Note that $\lambda_0 (\rho_0^{-1} - 1) + \lambda_1 (\rho_1^{-1} - 1)$ is not smaller than $\lambda_0 (1 - \rho_0) + \lambda_1 (1 - \rho_1) \ge \min(\lambda_0,\lambda_1) \delta > 0$. Therefore, we may replace $h$ in the previously displayed integrals with $h \big/ \bigl( \lambda_0 (\rho_0^{-1} - 1) + \lambda_1 (\rho_1^{-1} - 1) \bigr)$ and conclude that for measurable functions $h \ge 0$ on $\R$,
\[
	\int h \, dQ_* \ = \ \int h \tilde{g} \, dQ
\]
with
\[
	\tilde{g} \ := \ \begin{cases}
		\bigl( \lambda_0 (\rho_0^{-1} - 1) + \lambda_1 (\rho_1^{-1} - 1) \bigr)^{-1}
			& \text{on} \ A_* , \\
		0	& \text{on} \ \R \setminus A_* .
	\end{cases}
\]	
In other words, $dQ_*/dQ = \tilde{g}$, and $\tilde{g} \le \min(\lambda_0,\lambda_1)^{-1} \delta^{-1}$.

It remains to be shown that $\tilde{g}(y)$ may be replaced with the stated limit $g(y)$. Note first that by Lemma~\ref{Lem:Left.support.Q}, the set $\R \setminus \supp_{\rm left}(Q_*)$ has probability $0$ under $Q_*$, and $g(y) = 0$ for all $y \in \R\setminus \supp_{\rm left}(Q_*)$, because $Q_*((x,y]) = 0$ if $x < y$ is sufficiently close to $y$. It suffices to show that $g(y) = \tilde{g}(y)$ for any $y \in \supp_{\rm left}(Q_*) \subset \supp_{\rm left}(Q)$. If $\rho_0$ and $1 - \rho_1$ are constructed as in the proof of Theorem~\ref{Thm:EquivDef} by means of Lemma~\ref{Lem:IRN}, then for $j = 0,1$,
\[
	\rho_j(x,y) := \frac{Q_*((x,y])}{(Q_j + Q_*)((x,y])} \
	\to \ \rho_j(y) \quad\text{as} \ x \to y\,- ,
\]
and
\[
	Q_j((x,y]) \ = \ (\rho_j(x,y)^{-1} - 1) Q_*((x,y]) .
\]
Consequently, $Q((x,y]) = \bigl( \lambda_0 (\rho_0(x,y)^{-1} - 1) + \lambda_1 (\rho_1(x,y)^{-1} - 1) \bigr) Q_*((x,y])$, and this leads to
\begin{align*}
	\frac{Q_*((x,y])}{Q((x,y])} \
	&= \ \bigl( \lambda_0 (\rho_0(x,y)^{-1} - 1)
		+ \lambda_1 (\rho_1(x,y)^{-1} - 1) \bigr)^{-1} \\
	&\to \ \begin{cases}
		0
			& \text{if} \ y \not\in A_* \\
		\bigl( \lambda_0 (\rho_0(y)^{-1} - 1) + \lambda_1 (\rho_1(y)^{-1} - 1) \bigr)^{-1}
			& \text{if} \ y \in A_*
		\end{cases} \\
	&= \ \tilde{g}(y)
\end{align*}
as $x \to y\,-$.
\end{proof}

\begin{proof}[\bf Proof of Theorem~\ref{Thm:Weak.convergence}]
Note first that by the continuous mapping theorem, the marginal distributions $P_n = R_n(\cdot \times \R)$ converge weakly to $P$. Consequently, the Portmanteau theorem implies that $\liminf_{n\to \infty} P_n(U) \ge P(U)$ for any open set $U \subset \R$, while $\limsup_{n\to\infty} R_n(A) \le R(A)$ for any closed set $A \subset \R\times\R$. Specifically, taking $U = (-\infty,x_1)$ and $U = (x_2,\infty)$ shows that $x_1 < x_2$ are interior points of $\XX_n = \bigl\{ x\in\R\colon P_n((-\infty,x]), P([x,\infty)) > 0 \bigr\}$ for sufficiently large $n$.

For symmetry reasons, it suffices to show that $\liminf_{n\to \infty} q_n(\beta \,|\, x_2) \ge y$ for any fixed $y < \qw(\beta \,|\, x_1)$. The definition of $\qw(\beta \,|\, x_1)$ and the construction of $\Kw$ imply that
\[
	\beta \
	> \ \Kw(x_1,(-\infty,y]) \
	= \ \lim_{a \to a_o-} \frac{R \bigl( (a,x_1] \times (-\infty,y] \bigr)}{P((a,x_1])}
\]
with $a_o := \min\{a \le x_1 : P((a,x_1]) = 0\}$. Hence, there exists a number $a < a_o$ such that $P(\{a\}) = 0$, $P((a,x_1]) > 0$ and $R \bigl( (a,x_1] \times (-\infty,y] \bigr) / P((a,x_1]) < \beta$. Consequently, for any fixed $b \in (x_1,x_2)$ and sufficiently large $n$, the stochastic ordering of $K_n(x,\cdot)$ implies that
\begin{align*}
	\limsup_{n\to\infty} K_n(x_2,(-\infty,y]) \
	&\le \ \liminf_{n\to\infty}
		\frac{R_n \bigl( (a,b] \times (-\infty,y] \bigr)}{P_n((a,b])} \\
	&\le \ \liminf_{n\to\infty}
		\frac{R_n \bigl( [a,b] \times (-\infty,y] \bigr)}{P_n((a,b))} \\
	&\le \ \frac{R \bigl( [a,b] \times (-\infty,y] \bigr)}{P((a,b))}
		\ = \ \frac{R \bigl( (a,b] \times (-\infty,y] \bigr)}{P(a,b)} .
\end{align*}
But the right-hand side converges to $R \bigl( (a,x_1] \times (-\infty,y] \bigr)/P((a,x_1]) < \beta$ as $b \downarrow x_1$, and this implies that $q_n(\beta \,|\, x_1) > y$ for sufficiently large $n$.
\end{proof}

\begin{proof}[\bf Proof of Corollary~\ref{Cor:Weak.convergence}]
Note first that in general, $\qw(\beta \,|\, \cdot)$ and $\qe(\beta \,|\, \cdot)$ are left- and right-cont\-inu\-ous on $\XX$, respectively. For symmetry reasons, it suffices to verify this for $\qw(\beta \,|\, \cdot)$. Let $x \in \XX$ such that $x > \inf(\XX)$. Since $\qw(\beta \,|\, \cdot)$ is increasing on $\XX$, it suffices to show that $\liminf_{a \to x-} \qw(\beta \,|\, a) \ge \qw(\beta \,|\, x)$. If $P((a,x]) = 0$ for some $a < x$, then $\Kw(a',\cdot)$ and thus $\qw(\beta \,|\, a')$ are constant in $a' \in [a,x)$. If $P((a,x]) > 0$ for all $a < x$, then for any $y \in \R$,
\[
	\Kw(x,(-\infty,y]) \
	= \ \lim_{a \to x-} P((a,x])^{-1} \int_{(a,x]} \Kw(t,(-\infty,y]) \, P(\d t) ,
\]
and since $\Kw(t,(-\infty,y])$ is decreasing in $t \in \XX$, the right hand side equals
\[
	\lim_{a \to x-} \Kw(a,(-\infty,y]) .
\]
For any $y < \qw(\beta \,|\, x)$, this shows that $\Kw(a,(-\infty,y]) < \beta$ if $a < x$ is sufficiently close to $x$, whence $\qw(\beta \,|\, a) > y$. This implies that $\liminf_{a \to x-} \qw(\beta \,|\, a) = \qw(\beta \,|\, x)$.

If $\qw(\beta \,|\, \cdot) = \qe(\beta \,|\, \cdot)$ on $(a,b) \subset \XX$, then $\qw(\beta \,|\, \cdot)$ is continuous on $(a,b)$. But then one can easily deduce from Theorem~\ref{Thm:Weak.convergence} that $\lim_{n \to \infty} q_n(\beta \,|\, x) = \qw(\beta \,|\, x)$ for any $x \in (a,b)$. But for any fixed interval $[a',b'] \subset (a,b)$ and sufficiently large $n$, the functions $q_n(\beta \,|\, \cdot)$ are increasing on $[a',b']$, and it is well-known that the pointwise convergence of monotone functions on a compact interval to a continuous function implies uniform convergence.
\end{proof}

\begin{proof}[\bf Proof of Lemma~\ref{Lem:Minimum.Kuiper.distance}]
Let $x_1 < \cdots < x_{\ell}$ and $y_1 < \cdots < y_m$ be real numbers such that the support of $\hat{R}$ is contained in $\{(x_j,y_k) \colon 1\le j\le \ell, 1 \le k\le m\}$. With $x_0 := y_0 := -\infty$ and $x_{\ell+1} := y_{m+1} := \infty$, one can easily show that with the family $\mathcal{S} := \bigl\{ (a_1,a_2] \times (b_1,b_2] : a_1 < a_2, b_1 < b_2 \bigr\}$, for an arbitrary distribution $\tilde{R}$ on $\R\times\R$,
\begin{align*}
	\sup_{S \in \mathcal{S}} \, (\hat{R} - \tilde{R})(S) \
	&= \ \max_{S \in \overline{\mathcal{S}}} \, (\hat{R} - \tilde{R})(S) , \\
\intertext{and}
	\sup_{S \in \mathcal{S}} \, (\tilde{R} - \hat{R})(S) \
	&= \ \max_{S \in \underline{\mathcal{S}}} \, (\tilde{R} - \hat{R})(S) ,
\end{align*}
where
\begin{align*}
	\overline{\mathcal{S}} \
	&:= \ \bigl\{ [x_{j_1}, x_{j_2}] \times [y_{k_1}, y_{k_2}] :
		1 \le j_1 \le j_2 \le \ell, 1 \le k_1 \le k_2 \le m \bigr\} , \\
	\underline{\mathcal{S}} \
	&:= \ \bigl\{ (x_{j_1}, x_{j_2}) \times (y_{k_1}, y_{k_2}) :
		0 \le j_1 < j_2 \le \ell+1, 0 \le k_1 < k_2 \le m+1 \bigr\} .
\end{align*}
Now choose real numbers $x_1' < x_2' < \cdots < x_{2\ell+1}'$ and $y_1' < y_2' < \cdots < y_{2m+1}'$ with $x_{2j}' = x_j^{}$ for $1 \le j \le \ell$ and $y_{2k}' = y_k^{}$ for $1 \le k \le m$. Then $\|\tilde{R} - \hat{R}\|_{\rm K}$ does not change if we replace $\tilde{R}$ with the discrete distribution
\[
	\sum_{j=1}^{2\ell+1} \sum_{k=1}^{2m+1} p_{jk}^{} \delta_{(x_j',y_k')}^{} ,
\]
where
\begin{align*}
	p_{2j,2k}^{} \
	&:= \ \tilde{R} \bigl( \{x_j\} \times\{y_k\} \bigr) \quad
		\text{for} \ 1 \le j \le \ell, 1 \le k \le m , \\
	p_{2j-1,2k}^{} \
	&:= \ \tilde{R} \bigl( (x_{j-1},x_j) \times \{y_k\} \bigr) \quad
		\text{for} \ 1 \le j \le \ell+1, 1 \le k \le m , \\
	p_{2j,2k-1}^{} \
	&:= \ \tilde{R} \bigl( \{x_j\} \times (y_{k-1},y_k) \bigr) \quad
		\text{for} \ 1 \le j \le \ell, 1 \le k \le m+1 , \\
	p_{2j-1,2k-1}^{} \
	&:= \ \tilde{R} \bigl( (x_{j-1},x_j) \times (y_{k-1},y_k) \bigr) \quad
		\text{for} \ 1 \le j \le \ell+1, 1 \le k \le m+1 .
\end{align*}
Note that the values $\tilde{R}(S)$, $S \in \overline{\mathcal{S}} \cup \underline{\mathcal{S}}$, do not change when we perform this discretization of $\tilde{R}$. Moreover, if $\tilde{R}$ is TP2, then this discretized version is TP2 as well. Consequently, to minimize $\|\tilde{R} - \hat{R}\|_{\rm K}$ over all TP2 distributions $\tilde{R}$, we may restrict our attention to all TP2 distributions on the finite grid $\{x_1',\ldots,x_{2\ell+1}'\} \times \{y_1',\ldots,y_{2m+1}'\}$. Then $\|\tilde{R} - \hat{R}_{\rm emp}\|_{\rm K}$ is a continuous function of the corresponding matrix $\boldsymbol{p} = (p_{jk})_{j,k} \in [0,1]^{(2\ell+1)\times(2m+1)}$, and the constraints on $\tilde{R}$ define a compact subset of $[0,1]^{(2\ell+1)\times(2m+1)}$. Consequently, a minimizer does exist.
\end{proof}

\paragraph{Acknowledgements.}
This work was supported by Swiss National Science Foundation. We thank Johanna Ziegel and Alexander Jordan for stimulating discussions and Tilmann Gneiting for pointing us to the connection between likelihood ratio order and concavity of ROC curves. We are also grateful for constructive comments of two referees.



\vfill

\parbox[t]{0.49\textwidth}{Lutz D\"umbgen\\
Department of Mathematics and Statistics\\
University of Bern\\
Alpeneggstrasse 22\\
CH-3012 Bern\\
Switzerland\\[0.5ex]
lutz.duembgen@unibe.ch}
\hfill
\parbox[t]{0.49\textwidth}{Alexandre M\"osching\\
Nonclinical Biostatistics\\
F.\ Hoffmann-La Roche Ltd\\
Grenzacherstrasse 124\\
CH-4058 Basel\\
Switzerland\\[0.5ex]
alexandre.moesching@roche.com}

\end{document}